\documentclass{amsart}

\usepackage{amssymb}
\usepackage{amsfonts}
\usepackage{amsthm}
\usepackage{amsmath}
\usepackage{mathtools}
\usepackage{bbm}
\usepackage{bussproofs}
\usepackage{mathrsfs}
\usepackage{tikz, tikz-cd}
\usepackage[all,2cell]{xy}
\usepackage{ stmaryrd } 
\usepackage{todonotes}
\usepackage{hyperref}

\usepackage[backend=biber, style=numeric, maxbibnames=6]{biblatex}
\tikzset{%
	symbol/.style={%
		draw=none,
		every to/.append style={%
			edge node={node [sloped, allow upside down, auto=false]{$#1$}}}
	}
}

\usetikzlibrary{matrix,arrows}

\newtheorem{Theorem}{Theorem}

\newtheorem{proposition}[Theorem]{Proposition}
\newtheorem{lemma}[Theorem]{Lemma}
\newtheorem{corollary}[Theorem]{Corollary}

\newtheorem{thmx}{Theorem}
\newtheorem{propx}[thmx]{Proposition}

\newtheorem{corx}[thmx]{Corollary}

\theoremstyle{definition}

\newtheorem{remark}[Theorem]{Remark}
\newtheorem{definition}[Theorem]{Definition}

\DeclareMathOperator{\Ascr}{\mathscr{A}}

\DeclareMathOperator{\Cscr}{\mathscr{C}}

\DeclareMathOperator{\Fscr}{\mathscr{F}}
\DeclareMathOperator{\Gscr}{\mathscr{G}}

\DeclareMathOperator{\Pscr}{\mathscr{P}}

\DeclareMathOperator{\Tscr}{\mathscr{T}}

\DeclareMathOperator{\Fbb}{\mathbb{F}}

\DeclareMathOperator{\id}{id}

\DeclareMathOperator{\C}{\mathbb{C}}
\DeclareMathOperator{\Z}{\mathbb{Z}}
\DeclareMathOperator{\N}{\mathbb{N}}

\DeclareMathOperator{\KVect}{\mathnormal{K}-\mathbf{Vect}}

\DeclareMathOperator{\RMod}{\mathnormal{R}-\mathbf{Mod}}

\DeclareMathOperator{\Q}{\mathbb{Q}}

\DeclareMathOperator{\Ext}{Ext}

\DeclareMathOperator{\op}{op}

\DeclareMathOperator{\Per}{\mathbf{Perv}}

\DeclareMathOperator{\incl}{incl}

\DeclareMathOperator{\Sf}{\mathbf{Sf}}

\DeclareMathOperator{\Ch}{\mathbf{Ch}}

\DeclareMathOperator{\fCat}{\mathfrak{Cat}}

\DeclareMathOperator{\SfResl}{\mathbf{SfResl}}
\DeclareMathOperator{\quo}{\mathsf{quot}}

\DeclareMathOperator{\PC}{\mathbf{PC}}

\DeclareMathOperator{\Var}{\mathbf{Var}}

\DeclareMathOperator{\fil}{fil}
\DeclareMathOperator{\rea}{real}

\let\epsilon\varepsilon

\newcommand{\Dbeqstack}[1]{\mathnormal{D}^{\mathnormal{b}}_{\text{eq}}(\underline{[\mathnormal{G} \backslash \mathnormal{X}]}_{\bullet};\overline{\Q}_{\ell})}

\DeclareMathOperator{\of}{\overline{\mathnormal{f}}}

\DeclareMathOperator{\Bicat}{\mathsf{Bicat}}
\DeclareMathOperator{\One}{\mathbbm{1}}
\DeclareMathOperator{\cnst}{cnst}

\DeclareMathOperator{\gr}{gr}

\let\epsilon\varepsilon

\newcommand{\quot}[2]{\,_{#2}{#1}}

\newcommand{\ul}[1]{\underline{{#1}}}

\UseTwocells

\numberwithin{Theorem}{section}
\numberwithin{equation}{section}

\title{On the Equivariant Derived Category of Perverse Sheaves}
\author{Geoff Vooys}
\date{\today}

\begin{document}

\begin{abstract}
In this paper we extend Beilinson's realization formalism for triangulated categories and filtered triangulated categories to a pseudofunctorial and pseudonatural setting. As a consequence we prove an equivariant version of Beilinson's Theorem: for any algebraic group $G$ over an algebraically closed field $K$ and for any $G$-variety $X$, there is an equivalence of categories $D_G^b(X; \overline{\mathbb{Q}}_{\ell}) \simeq D_G^b(\mathbf{Perv}(X;\overline{\mathbb{Q}}_{\ell}))$ where $\ell$ is an integer prime coprime to the characteristic of $K$. We also show that the equivariant analogues of the other non-$D$-module aspects of Beilinson's Theorem hold in the equivariant case.
\end{abstract}


\maketitle
\tableofcontents
	
\section{Introduction}

In the 1980's Beilinson proved in \cite{Beilinson} a result of crucial importance in both representation theory and arithmetic algebraic geometry: there is an equivalence of categories
\[
D^b(\Per(X;\overline{\Q}_{\ell})) \simeq D_c^b(X;\overline{\Q}_{\ell})
\]
for all separated finite type schemes over an algebraically closed field $K$ with $\ell$ coprime to the characteristic of $K$. This celebrated result has had myriad ramifications and led to applications which allow us to compute extensions of perverse sheaves in the derived category as opposed to in the derived category of perverse sheaves, which is often a crucial step in actually being able to calculate things because we have access to the usual (derived) functors which are well-understood. This is also important for $K$-theoretic results, as it gives rise to an isomorphism of Grothendieck groups $\mathsf{K}_0D^b(\Per(X)) \cong \mathsf{K}_0D^b(X)$, which is helpful in testing conjectures which appear in representation theory. 

As representation theory and the Langlands Programme have proceeded over time, there has been a need to develop more of the theory of equivariant derived categories and equivariant algebraic geometry which is applicable to the study of $p$-adic groups. In \cite{CFMMX} the authors introduced a microlocal approach towards the study of A-packets motivated by \cite{ABV} which use, in fundamental ways, the theory of equivariant derived categories, equivariant perverse sheaves, and equivariant nearby and vanishing cycles. These techniques have been further used in \cite{cunningham2023proof} and related work. These insights highlight the importance of the equivariant derived category and equivariant perverse sheaves in the study of $p$-adic groups and the Local Langlands Programme for $p$-adic groups.

Alternatively, the study of much of the pure homological algebra which goes into proving Beilinson's Theorem \cite[Theorem 1.3]{Beilinson} is of independent interest. The techniques used require a careful manipulation of triangulated and truncation structures, and so admit applications beyond algebro-geometric and representation-theoretic circumstances. It is shown in \cite{Monograph}, with an eye towards these further applications, that much of the category theory which equivariant descent comes down to the language of pseudofunctors, pseudocones, and higher-category theory. In fact, the equivariant derived category $D_G^b(X)$ is the pseudolimit in $\fCat$ of a certain pseudofunctor. The various properties of the equivariant derived category we care about in practice (such as its triangulation, standard $t$-structure, perverse $t$-structure, six functor formalism, etc.\@) follow directly from various properties of the pseudofunctor indexing the diagram for which $D_G^b(X)$ is a pseudolimit. Because the pseudocone and pseudolimit language is so fruitful in working with the equivariant derived category, we show this perspective is well-suited towards generalizing the filtered triangulated categories of \cite{Beilinson} to certain classes of pseudofunctors and also towards using this general framework to study equivariant algebraic geometry. It is also worth noting that this pseudofunctorial approach is an alternative to the averaging functors used in \cite{PramodBook}, which showed that this equivalence occurs when $G$ is a finite complex algebraic group and $X$ is a $G$-variety over $\C$.

\subsection{Main Results and Structure of the Paper}

The main result of the paper is Theorem \ref{Thm: Equivariant Beilinson} below, which shows many cases in which the equivariant bounded derived category of perverse sheaves is equivalent to the equivariant bounded derived category.
\begin{thmx}[Cf.\@ Theorem \ref{Thm: Equivariant Beilinson}]
The following equivalences of categories hold:
\begin{enumerate}
	\item Let $K$ be a field, $G$ a smooth algebraic group over $K$, and $X$ a left $G$-variety. If $R$ is any finite ring with characteristic coprime to the characteristic of $K$ then there is an equivalence of categories
	\[
	D^b_G(X;\RMod) \simeq D^b_G(\Per(X;\RMod)
	\]
	where $\Per(X;\RMod)$ denotes the category of perverse sheaves with coefficients in $\RMod$, $D^b_G(\Per(X;\RMod))$ is the equivariant derived category of perverse sheaves of $R$-modules.
	\item Let $K$ be an algebraically closed field, $G$ an algebraic group over $K$, and $X$ a left $G$-variety. Then
	\[
	D_G^b(X;\overline{\Q}_{\ell}) \simeq D_G^b(\Per(X;\overline{\Q}_{\ell})).
	\]
	\item If $G$ is a complex algebraic group and $X$ is a left $G$-variety over $\C$ then for any field $K$ there is an equivalence
	\[
	D_G^b(X(\C),\KVect) \simeq D_G^b(\Per(X(\C),\KVect)).
	\]
	\item If $K = \Fbb_q$ is a finite field, $G$ is an algebraic group over $K$, and $X$ is a left $G$-variety then for any prime $\ell$ coprime to $q$,
	\[
	D_{G,m}^b(X,\overline{\Q}_{\ell}) \simeq D_G^b(\Per_{m}(X;\overline{\Q}_{\ell})),
	\]
	where $D_{G,m}^b(X,\overline{\Q}_{\ell})$ is the equivariant derived category of mixed $\ell$-adic sheaves (and similarly in the perverse case).
\end{enumerate}
\end{thmx}
\begin{corx}[Cf.\@ Corollary \ref{Cor: Section Equiv Beilin: EDC is EDPer for all var}]
Let $K$ be a field, $G$ a smooth algebraic group over $K$, and let $X$ be a left $G$-variety. Then there is an equivalence of categories
\[
D_G^b(X;\overline{\Q}_{\ell}) \simeq D^b_G\left(\Per(X;\overline{\Q}_{\ell})\right).
\]
\end{corx}

We prove this primarily by extending the language of Beilinson's filtered triangulated categories and realization functor to pseudofunctors and pseudocone categories. In Section \ref{Section: } we give a review of pseudocone categories and highlight some terminology, results, and other background material used throughout this paper. Afterwards, in Section \ref{Section: Pseudocone Realization Functors}, we begin the meat of our work and develop the extension of Beilinson's realization functors to the pseudofunctor setting. This long section culminates in the two main theorems below.
\begin{thmx}[Cf.\@ Theorem \ref{Thm: Section Pseudocone Realizations: Pseudofunctor Realization}]
Let $F:\Cscr^{\op} \to \fCat$ be a filtered triangulated pseudofunctor over triangulated pseudofunctor $T:\Cscr^{\op} \to \fCat$ and assume that $T$ and $F$ are truncated with compatible $t$-structures. Then if $Q$ is the quasi-isomorphism subpseudofunctor of $\Ch^b(T^{\heartsuit})$, there is a unique pseudonatural transformation $\ul{\rea}:Q^{-1}\Ch^b(T^{\heartsuit}) \to T$ for which the diagram
\[
\begin{tikzcd}
	\Ch^b(T^{\heartsuit}) \ar[rr]{}{\widetilde{\rea}} \ar[dr, swap]{}{\ul{\lambda}_Q} & & T \\
	& Q^{-1}\Ch^b\left(T^{\heartsuit}\right) \ar[ur, swap]{}{\ul{\rea}}
\end{tikzcd}
\]
commutes in $\Bicat(\Cscr^{\op},\fCat)$. Furthermore, each object functor $\quot{\ul{\rea}}{X}$ is $t$-exact for all obects $X$ of $\Cscr$.
\end{thmx}
\begin{thmx}[Cf.\@ Theorem \ref{Thm: Section Pseudocone Realization: Pseudocone Realization}]
Let $F:\Cscr^{\op} \to \fCat$ be a filtered triangulated pseudofunctor over triangulated pseudofunctor $T:\Cscr^{\op} \to \fCat$ and assume that $T$ and $F$ are truncated with compatible $t$-structures. Then if $Q$ is the quasi-isomorphism subpseudofunctor of $\Ch^b(T^{\heartsuit})$, there is a unique pseudonatural transformation $\ul{\rea}:Q^{-1}\Ch^b(T^{\heartsuit}) \to T$ for which the diagram
\[
\begin{tikzcd}
	\PC(\Ch^b(T^{\heartsuit})) \ar[rr]{}{\PC(\widetilde{\rea})} \ar[dr, swap]{}{\PC(\ul{\lambda}_Q)} & & \PC(T) \\
	& \PC(D^b\left(T^{\heartsuit}\right)) \ar[ur, swap]{}{\PC(\ul{\rea})}
\end{tikzcd}
\]
commutes in $\Bicat(\Cscr^{\op},\fCat)$ with $\PC(\ul{\rea})$ $t$-exact. Furthermore, for any obejct $A$ of $\PC(\Ch^b(T^{\heartsuit}))$,
\[
\PC(\ul{\rea})(A) = \left\lbrace \quot{\ul{\rea}}{X}(\quot{A}{X}) \; : \; X \in \Cscr_0\right\rbrace
\]
where $\quot{\ul{\rea}}{X}$ is the realization $D^b(T(X)^{\heartsuit}) \to T(X).$
\end{thmx}

Finally we conclude the paper with our main theorem and corollary before giving a result on extension groups of equivariant perverse sheaves.

\begin{propx}[Cf.\@ Proposition \ref{Prop: The Exty boi}]
Let $\Fscr$ and $\Gscr$ be two equivariant perverse sehaves and let $n \in \Z$. Then there is an isomorphism
\[
\Ext^n_{\Per_G(X;\overline{\Q}_{\ell})}(\Fscr,\Gscr) \cong \Ext^n_{D_G^b(X;\overline{\Q}_{\ell})}(\Fscr,\Gscr).
\]
\end{propx}

\subsection{Acknowledgments}
I would like to thank Clifton Cunningham and James Steele for many stimulating conversations and posing the questions ``is it true that there is an equivalence
\[
D_G^b\left(\Per(X;\overline{\Q}_{\ell})\right) \simeq D_G^b(X;\overline{\Q}_{\ell})
\]
for a complex algebraic group $G$ and $G$-variety $X$?'' and ``for any $n \in \Z$ is it true that
\[
\Ext^n_{\Per_G(X;\overline{\Q}_{\ell})}(\Fscr,\Gscr) \cong D^b_G(X;\overline{\Q}_{\ell})(\Fscr,\Gscr[n])
\]
for equivariant perverse sheaves $\Fscr, \Gscr$?'' to me which motivated writing this paper. I am additionally grateful to the entire Voganish research group for helpful discussions and input.

\section{Equivariant Categories, Descent, and Pseudocones}\label{Section: }
In this section we recall the basics of the language of categories of pseudocones, some of their properties which we will need, and how they apply to the categories we study. A systematic and in-depth study of these categories, their properties, and how they apply to equivariant descent for varieties and topological spaces is done in \cite{Monograph}. What is important for this paper is that the work there allows us to construct an equivariant derived category $D_G^b(\Per(X))$ of perverse sheaves on $X$ which is a triangulated category with standard $t$-structure for which $D^b_G(\Per(X))^{\heartsuit_{\operatorname{stand}}} \simeq \Per_G(X)$.

Before diving in directly, let us set some conventions. In this paper all the schemes we consider are quasi-projective and by varieties we mean separated schemes of finite type over a field. We also assume some familarity with $2$-category theory, the theory of pseudofunctors, pseudonatural transformations, modifications, and some of the basics involving the theory of bicategories. A textbook introduction to the higher category theory we use in this paper can be found in \cite{TwoDimCat}.

We will begin by discussing some of the basics of equivariant descent we use to define the equivariant derived category and equivariant derived category of perverse sheaves. The main technique used to define the equivariant category on a variety follows the definitions of Lusztig and Achar: we must descend by principal $G$-varieties (which, in case $G$ is not affine, are locally isotrivial). An in-depth discussion of this condition is given in \cite[{Section 6.1}]{Monograph} but for the reader who is content with working over affine algebraic groups this poses no issue nor distinction.

\begin{definition}\label{Defn: Pseudocone cat}
Let $F:\Cscr^{\op} \to \fCat$ be a pseudofunctor. The category of pseudocones of shape $F$ (with apex $\One$) is defined to be the category
\[
\PC(F) := \Bicat(\Cscr^{\op},\fCat)(\cnst(\One),F)
\]
where $\cnst(\One):\Cscr^{\op} \to \fCat$ is the constant pseudofunctor at the terminal category $\One$.
\end{definition}
A large list of these pseudocone categories is given in \cite[Example 2.3.12]{Monograph} and many of the categories given come from algebraic-geometric contexts. We now recall from \cite[Remark 2.3.8]{Monograph} a more explicit combinatorial description of the objects and morphisms in a pseudocone category.
\begin{remark}\label{Remark: Pseudocone Section: Pseudocone explicit} 
	Let us unwrap the definition of $\PC(F)$ in order to give an explicit description of the objects and morphisms of $\PC(F)$, as working with $\PC(F)$ is much more straightforward with this knowledge. To this end recall that to give a functor $\quot{\alpha}{X}:\One \to F(X)$ is to pick out an object of $F(X)$ (and its identity morphism, of course). Similarly, if $f:X \to Y$ is a morphism in $\Cscr$ then the $2$-cell 
	\[
	\quot{\alpha}{f}:(\quot{\alpha}{X} \circ \id_{\One})(\ast) \xrightarrow{\cong} (F(f) \circ \quot{\alpha}{Y})(\ast)
	\]
	corresponds to giving an isomorphism $\quot{\alpha}{X}(\ast) \cong F(f)(\quot{\alpha}{Y}(\ast))$; we write 
	\[
	\tau_f:F(f)(\quot{\alpha}{Y}(\ast)) \xrightarrow{\cong} \quot{\alpha}{X}
	\] 
	for the corresponding inverse isomorphism and call these the transition isomorphisms (of the pseudocone) --- note that since the inverse of an isomorphism is unique, determining the $\quot{\alpha}{f}$ is equivalent to determining the $\tau_f$. Asking the pasting diagrams to coincide is equivalent to asking that for a composable pair $X \xrightarrow{f} Y \xrightarrow{g} Z$ of morphisms in $\Cscr$, the diagram
	\[
	\begin{tikzcd}
		F(f)\left(F(g)\quot{\alpha}{Z}(\ast)\right) \ar[d, swap]{}{\phi_{f,g}^{\quot{\alpha}{Z}(\ast)}} \ar[rr]{}{F(f)(\tau_g)} & & F(f)(\quot{\alpha}{Y}(\ast)) \ar[d]{}{\tau_f} \\
		F(g \circ f)\quot{\alpha}{Z}(\ast) \ar[rr, swap]{}{\tau_{g \circ f}} & & \quot{\alpha}{X}(\ast)
	\end{tikzcd}
	\]
	must commute. That is, the pseudonaturality condition gives rise to the cocycle condition 
	\[
	\tau_{g \circ f} \circ \phi_{f,g}^{\quot{\alpha}{Z}(\ast)} = \tau_{f} \circ F(f)(\tau_g).
	\]
	Similarly, if we have pseudocones $\alpha, \beta:\cnst(\One) \Rightarrow F$ a modification $\rho:\alpha \Rightarrow \beta$ is given by a family of natural transformations $\quot{\rho}{X}:\quot{\alpha}{X} \Rightarrow \quot{\beta}{X}$, one for each $X \in \Cscr_0$, such that for all morphisms $f:X \to Y$ in $\Cscr$ the diagrams of functors and natural transformations
	\[
	\begin{tikzcd}
		\quot{\alpha}{X} \ar[d, swap]{}{\quot{\alpha}{f}} \ar[rr]{}{\quot{\rho}{X}} & & \quot{\beta}{X} \ar[d]{}{\quot{\beta}{f}} \\
		F(f) \circ \quot{\alpha}{Y} \ar[rr, swap]{}{F(f) \ast \quot{\rho}{Y}} & & F(f) \circ \quot{\beta}{Y}
	\end{tikzcd}
	\]
	commute. However, this implies that we can replace the isomorphisms $\quot{\alpha}{f}$ with the corresponding transition isomorphisms to deduce that the diagrams
	\[
	\begin{tikzcd}
		F(f) \circ \quot{\alpha}{Y} \ar[rr]{}{F(f) \ast \quot{\rho}{Y}} \ar[d, swap]{}{\tau_f} & & F(f) \circ \quot{\beta}{Y} \ar[d]{}{\sigma_f} \\
		\quot{\alpha}{X} \ar[rr, swap]{}{\quot{\rho}{X}} & & \quot{\beta}{X}
	\end{tikzcd}
	\]
	also must commute. This allows us to give the following explicit description for $\PC(F)$:
	\begin{itemize}
		\item Objects: Pairs $(A, T_A)$ where 
		\[
		A = \lbrace \quot{A}{X} \; | \; X \in \Cscr_0, \quot{A}{X} \in F(X)_0 \rbrace
		\] 
		is a collection of objects and transition isomorphisms
		\[
		T_A = \lbrace \tau_f^{A}:F(F)(\quot{A}{Y}) \xrightarrow{\cong} \quot{A}{X} \; | \; f:X \to Y, f \in \Cscr_1 \rbrace
		\]
		which satisfy the cocycle condition
		\[
		\tau_{g \circ f}^A \circ \phi_{f,g} = \tau_f^A \circ F(f)\left(\tau_g^A\right)
		\]
		for any composable morphisms $X \xrightarrow{f} Y \xrightarrow{g} Z$ in $\Cscr$.
		\item Morphisms: A morphism $P:(A,T_A) \to (B, T_B)$ is a collection of morphisms 
		\[
		P = \lbrace \quot{\rho}{X}:\quot{A}{X} \to \quot{B}{X} \; | \; X \in \Cscr_0 \rbrace
		\]
		such that for any morphisms $f:X \to Y$ the diagram
		\[
		\begin{tikzcd}
			F(F)(\quot{A}{X}) \ar[rr]{}{F(f)(\quot{\rho}{X})} \ar[d, swap]{}{\tau_f^A} & & F(f)(\quot{B}{Y}) \ar[d]{}{\tau_f^B} \\
			\quot{A}{X} \ar[rr, swap]{}{\quot{\rho}{X}}	& & \quot{B}{X}
		\end{tikzcd}
		\]
		commutes.
	\end{itemize}
\end{remark}
 
Let us now define the pseudofunctors which will induce the equivariant bounded derived category of perverse sheaves we consider in this paper; the core ideas we use go back to the perspectives pioneered and championed by \cite{BernLun}, \cite{LusztigCuspidal2}, \cite{PramodBook}. We first need to define the category which controls the equivariant descent which we use to define our equivariant derived categories. While the discussion here is quite brief, a detailed and in-depth discussion is given in \cite[Section 2.1, 6.1]{Monograph}. As a result, we recall what it means to be a smooth free $G$-variety, a smooth free $G$-resolution of a $G$-variety, and then describe the pseudofunctors which induce the equivariant derived categories of perverse sheaves and the usual equivariant derived category.

\begin{definition}[{\cite[Section 1.9]{LusztigCuspidal2}}]\label{Defn: Pseudocone Section: Smooth free varieties}
	A left $G$-variety $\Gamma$ is a smooth free $G$-variety if:
	\begin{itemize}
		\item $\Gamma$ is smooth, pure dimensional, and admits a geometric quotient $q:\Gamma \to G \backslash \Gamma$ (cf.\@ \cite[Definition 0.6]{GIT}).
		\item The geometric quotient 
		\[
		\quo:\Gamma \to G \backslash \Gamma
		\] 
		is ({\'e}tale) locally isotrivial with fibre $G$, i.e., there is an {\'e}tale cover $\lbrace \varphi_i:U_i \to G \backslash \Gamma \; | \; i \in I \rbrace$ of $G \backslash \Gamma$ such that in each pullback diagram
		\[
		\begin{tikzcd}
			\quo^{-1}(U_i) \ar[r] \ar[d] & \Gamma \ar[d]{}{q} \\
			U_i \ar[r, swap]{}{\varphi_i} & G \backslash \Gamma
		\end{tikzcd}
		\]
		there is an isomorphism $\quo^{-1}(U_i) \cong G \times U_i$, the corresponding map $G \times U_i \to U_i$ acts via projection $\pi_2:G \times U_i \to U_i$, and the {\'e}tale map $\varphi_i:U_i \to G \backslash \Gamma$ can be taken to be finite {\'e}tale.
	\end{itemize}
\end{definition}
\begin{remark}
	The reason we ask for the quotients above to be locally isotrivial is so that we can apply \cite[Proposition 6.1.4]{Monograph}, which states that for any smooth free $G$-variety $\Gamma$ and any left $G$-variety $X$ there is a geometric quotient $G \backslash (\Gamma \times X)$. When $G$ is affine this poses no issue; by \cite[Remark 6.1.5]{Monograph} (cf.\@ also \cite[Proposition 6.1.13]{PramodBook} and \cite[Lemme XIV 1.4]{Raynaud}) $\Sf(G)$-varieties are exactly those pure dimensional smooth free $G$-varieties which admit a geometric quotient.
\end{remark}
\begin{definition}[{\cite[Section 1.9]{LusztigCuspidal2}}]\label{Defn: Pseudocone Section: Smooth free morphisms}
	A morphism of smooth free $G$-varieties is a smooth $G$-equivariant morphism $f:\Gamma \to \Gamma^{\prime}$ where $\Gamma$ and $\Gamma^{\prime}$ are smooth free $G$-varieties and $f$ is a morphism of constant fibre dimension. Explicitly, for every $\gamma_0, \gamma_1 \in \lvert \Gamma \rvert$ we have $\dim f^{-1}(\gamma_0) = \dim f^{-1}(\gamma_1)$.
\end{definition}
\begin{definition}\label{Defn: Pseudocone Section: Sf(G)}
	Define the category of smooth free $G$-varieties $\Sf(G)$ as follows:
	\begin{itemize}
		\item Objects: Smooth free $G$-varieties $\Gamma$ (cf.\@ Definition \ref{Defn: Pseudocone Section: Smooth free varieties}).
		\item Morphisms: Morphisms of smooth free $G$-varieties (cf.\@ Definition \ref{Defn: Pseudocone Section: Smooth free morphisms}).
		\item Composition and Identities: As in $\Var_{/K}$.
	\end{itemize}
\end{definition}
The category of smooth free $G$-resolutions of a $G$-variety is then the category ``$\Sf(G) \times X$.'' We'll justify this as a good category with which to do equivariant descent shortly, but we'll get to know the category of resolutions itself first.
\begin{definition}\label{Defn: Pseudocone Section: SfResl}
	Let $X$ be a left $G$-variety. We define the category $\SfResl_G(X)$ of smooth free $G$-resolutions of $X$ as follows.
	\begin{itemize}
		\item Objects: Maps $\pi_2:\Gamma \times X \to X$ for $\Gamma \in \Sf(G)_0$.
		\item Morphisms: Maps $f:\Gamma \times X \to \Gamma^{\prime} \times X$ of the form $f = h \times \id_X$ for $h \in \Sf(G)(\Gamma, \Gamma^{\prime})$.
		\item Composition and Identities: As in $\Var_{/K}$.
	\end{itemize}
\end{definition}

The categories $\Sf(G)$ and $\SfResl_G(X)$ are studied in more detail in \cite[Chapter 6]{Monograph}. The important aspects we need to know about the category $\SfResl_G(X)$ are the following:
\begin{enumerate}
	\item For every smooth free $G$-variety $\Gamma \in \Sf(G)_0$ and every $G$-variety $X$, the variety $\Gamma \times X$ admits a geometric quotient variety $G \backslash (\Gamma \times X)$ which is functorial in $\SfResl_G(X)$ (cf.\@ \cite[Proposition 6.1.5]{Monograph} for the existence of the quotient functor; note that the local isotriviality of the resolving varieties $\Gamma$ and quasi-projectivity of $X$ are both essential). If $f \times \id_X:\Gamma \times X \to \Gamma^{\prime} \times X$ is a morphism in $\SfResl_G(X)$ then we write $\of:G \backslash (\Gamma \times X) \to G \backslash (\Gamma^{\prime} \times X)$
	\item If $X$ is a free $G$-variety then there is an equivalence of categories 
	\[
	D_G^b(X;\overline{\Q}_{\ell}) \simeq D_c^b(G \backslash X; \overline{\Q}_{\ell}).
	\]
	Note that this follows from \cite[Proposition 2.3.13]{Monograph}; cf.\@ also Proposition \ref{Prop: Sanity} below.
\end{enumerate}
\begin{definition}\label{Defn: Section Recall: Pseudofunctor of derived perverse sheaves}
Consider the following flavours of pseudofunctors of derived categories of perverse sheaves:
\begin{enumerate}
	\item Let $K$ be a field, $G$ a smooth algebraic group over $K$, $X$ a left $G$-variety, and let $R$ be a finite ring for which the characteristic of $R$ is coprime to the characteristic of $K$. Then we define the pseudofunctor 
	\[
	D^b(\Per(G \backslash(-));\RMod):\SfResl_G(X)^{\op} \to \fCat
	\] 
	as follows: for each object $\Gamma \times X$ of $\SfResl_G(X)$ we define the fibre category to be $D^b(\Per(G \backslash (\Gamma \times X);\RMod))$, where $\Per(G \backslash (\Gamma \times X);\RMod)$ is the category of perverse ({\'e}tale) sheaves of $R$-modules on the quotient variety $G \backslash (\Gamma \times X)$, and for each morphism $f\times\id_X$ we define the fibre functor to be given by ${}^{p}(G \backslash (f \times \id_X))^{\ast}$.
	\item Let $K$ be a field, $G$ a smooth algebraic group over $K$, $X$ a left $G$-variety, and let $\ell$ be an integer prime which is coprime to the characteristic of $K$. Then we define the pseudofunctor
	\[
	D^b(\Per(G \backslash(-));\overline{\Q}_{\ell}):\SfResl_G(X)^{\op} \to \fCat
\] 
as follows: for each object $\Gamma \times X$ of $\SfResl_G(X)$ we define the fibre category to be $D^b(\Per(G \backslash (\Gamma \times X);\overline{\Q}_{\ell}))$, where $\Per(G \backslash (\Gamma \times X);\overline{\Q}_{\ell})$ is the category of perverse $\ell$-adic sheaves on the quotient variety $G \backslash (\Gamma \times X)$, and for each morphism $f\times\id_X$ we define the fibre functor to be given by ${}^{p}(G \backslash (f \times \id_X))^{\ast}$.	
	\item Let $G$ be a complex algebraic group, let $X$ be a left $G$-variety (over $\C$), and let $K$ be a field. We then define the psuedofunctor
		\[
	D^b\left(\Per((G \backslash(-))(\C)),\KVect\right):\SfResl_G(X)^{\op} \to \fCat
	\] 
	as follows: for each object $\Gamma \times X$ of $\SfResl_G(X)$ we define the fibre category to be $D^b(\Per((G \backslash (\Gamma \times X))(\C),\KVect))$, where the category $\Per((G \backslash (\Gamma \times X))(\C),\KVect)$ is the category of perverse sheaves of $K$-vector spaces in the classical topology on the quotient variety $G \backslash (\Gamma \times X)$, and for each morphism $f\times\id_X$ we define the fibre functor to be given by ${}^{p}(G \backslash (f \times \id_X))^{\ast}$.	
	\item Let $\Fbb_q$ be a finite field, let $G$ be an algebraic group over $\Fbb_q$, $X$ a left $G$-variety, and let $\ell$ be an integer prime with $\gcd(q,\ell) =1$. Then we define the pseudofunctor
		\[
	D^b(\Per_m(G \backslash(-));\overline{\Q}_{\ell}):\SfResl_G(X)^{\op} \to \fCat
	\] 
	as follows: for each object $\Gamma \times X$ of $\SfResl_G(X)$ we define the fibre category to be $D^b(\Per_m(G \backslash (\Gamma \times X);\overline{\Q}_{\ell}))$, where $\Per_m(G \backslash (\Gamma \times X);\RMod)$ is the category of mixed perverse $\ell$-adic sheaves on the quotient variety $G \backslash (\Gamma \times X)$, and for each morphism $f\times\id_X$ we define the fibre functor to be given by ${}^{p}(G \backslash (f \times \id_X))^{\ast}$.	
\end{enumerate}
\end{definition}
\begin{definition}\label{Defn: Section Recall: EDC of Perverse Sheaves}
%
The four main flavours of equivariant derived category we consider in this paper are given as follows:
\begin{enumerate}
	\item Let $K$ be a field, $G$ a smooth algebraic group over $K$, $X$ a left $G$-variety, and let $R$ be a finite ring with characteristic coprime to the characteristic of $K$. The equivariant bounded derived category of perverse sheaves of $R$-modules is given by
	\[
	D_G^b\left(\Per\left(X,\RMod\right)\right) := \PC\left(D^b(\Per(G \backslash (-),\RMod))\right)
	\]
	where $D^b(\Per(G \backslash (-),\RMod))$ is the pseudofunctor described in Flavour $(1)$ of Definition \ref{Defn: Section Recall: Pseudofunctor of derived perverse sheaves}.
	\item Let $K$ be a field, $G$ a smooth algebraic group over $K$, $X$ a left $G$-variety, and let $\ell$ be an integer prime which is coprime to the characteristic of $K$. Then the equivariant bounded derived category of $\ell$-adic perverse sheaves is defined to be
	\[
	D_G^b\left(\Per(X;\overline{\Q}_{\ell})\right) := \PC\left(D^b\left(\Per\left(G \backslash (-); \overline{\Q}_{\ell}\right)\right)\right)
	\]
	where $D^b(\Per(G \backslash (-);\overline{\Q}_{\ell}))$ is the pseudofunctor described in Flavour $(2)$ of \ref{Defn: Section Recall: Pseudofunctor of derived perverse sheaves}.
	\item  Let $G$ be a complex algebraic group, let $X$ be a left $G$-variety (over $\C$), and let $K$ be a field. We then define the equivariant bounded derived category of perverse $K$-sheaves in the classical topology on $X$ by
	\[
	D_G^b\left(X(\C),\KVect\right) := \PC\left(D^b\left(\Per((G \backslash(-))(\C)),\KVect\right)\right)
	\] 
	where $D^b(\Per((G \backslash (-))(\C), \KVect))$ is the pseudofunctor described in Flavour $(3)$ of Definition \ref{Defn: Section Recall: Pseudofunctor of derived perverse sheaves}.
	\item Let $\Fbb_q$ be a finite field, let $G$ be an algebraic group over $\Fbb_q$, $X$ a left $G$-variety, and let $\ell$ be an integer prime with $\gcd(q,\ell) =1$. Then we define the equivariant bounded derived category of mixed $\ell$-adic perverse sheaves to be
	\[
	D_G^b\left(\Per_m(X;\overline{\Q}_{\ell})\right) := \PC\left(D^b\left(\Per_m\left(G \backslash (-);\overline{\Q}_{\ell}\right)\right)\right)
	\]
	where $D^b(\Per_m(X;\overline{\Q}_{\ell}))$ is the pseudofunctor described in Flavour $(4)$ of Definition \ref{Defn: Section Recall: Pseudofunctor of derived perverse sheaves}.
\end{enumerate}
\end{definition}
\begin{definition}\label{Defn: Section Recall: EDC}
Consider the following different flavours of equivariant derived categories:
\begin{enumerate}
	\item Let $K$ be a field, $G$ a smooth algebraic group over $K$, and $X$ a left $G$-variety. Then the equivariant derived category of {\'e}tale sheaves on $X$, $D_G^b(X)$, is the category $\PC(D)$ where $D:\SfResl_G(X)^{\op} \to \fCat$ is the pseudofunctor which sends each variety $\Gamma \times X$ to the constructible derived category $D_c^b(G \backslash (\Gamma \times X))$, sends a morphism $f \times \id_X:\Gamma \times X \to \Gamma^{\prime} \times X$ to the pullback functor $(G \backslash (f \times \id_X))^{\ast}$, and has any fixed choice of compositor isomorphisms.
	\item Let $K$ be a field, $G$ a smooth algebraic group over $K$, and $X$ a left $G$-variety. Then the equivariant derived category of {\'e}tale sheaves of $R$-modules on $X$, $D_G^b(X;\RMod)$, is the category $\PC(D)$ where $D:\SfResl_G(X)^{\op} \to \fCat$ is the pseudofunctor which sends each variety $\Gamma \times X$ to the constructible derived category $D_c^b(G \backslash (\Gamma \times X);\RMod)$ of $R$-modules, sends a morphism $f \times \id_X:\Gamma \times X \to \Gamma^{\prime} \times X$ to the pullback functor $(G \backslash (f \times \id_X))^{\ast}$, and has any fixed choice of compositor isomorphisms.
	\item Let $K$ be an algebraically closed field, $G$ an algebraic group over $K$, and $X$ a left $G$-variety. Then the equivariant derived category of $\ell$-adic sheaves on $X$ for a prime $\ell$ coprime to the characteristic of $K$, $D_G^b(X;\overline{\Q}_{\ell})$, is the category $\PC(D)$ where $D:\SfResl_G(X)^{\op} \to \fCat$ is the pseudofunctor which sends each variety $\Gamma \times X$ to the constructible $\ell$-adic derived category $D_c^b(G \backslash (\Gamma \times X);\overline{\Q}_{\ell})$, sends a morphism $f \times \id_X:\Gamma \times X \to \Gamma^{\prime} \times X$ to the pullback functor $(G \backslash (f \times \id_X))^{\ast}$, and has any fixed choice of compositor isomorphisms.
	\item Let $G$ be a complex algebraic group, $X$ a left $G$-variety, and let $K$ be a field. Then the equivariant derived category of $K$-sheaves in the classical topology on $X$, $D_G^b(X(\C),\KVect)$ is given by $\PC(D)$ where $D:\SfResl_G(X)^{\op} \to \fCat$ is the pseudofunctor given by sending each variety $\Gamma \times X$ to the bounded derived category of $K$-sheaves
	\[
	D_c^b\left(\left(G \backslash (\Gamma \times X)\right)(\C),\KVect\right)
	\]
	with algebriacally constructible cohomology and sends each morphism $f\times \id_X$ to the corresponding fibre functor $(G \backslash (f \times \id_X))^{\ast}$.
\end{enumerate}
\end{definition}

As an immediate consequence of the definition of the equivariant derived category of perverse sheaves in all cases, we get the following equivalences of categories.
\begin{proposition}\label{Prop: Sanity}
Let $D^b_G(\Per(X))$ be any of the four flavours of equivariant derived cateogries of perverse sheaves given in Definition \ref{Defn: Section Recall: EDC of Perverse Sheaves}. Then if $X$ is a free $G$-variety in the sense that the $G$-action on $X$ is locally isotrivial there is an equivalence of categories
\[
D_G^b(\Per(X)) \simeq D^b(\Per(G \backslash X)).
\]
\end{proposition}
\begin{proof}
Because the action of $G$ on $X$ is locally isotrivial, $X$ is a terminal object in $\SfResl_G(X)$. The proposition now follows from \cite[Proposition 2.3.13]{Monograph} and the fact that if $D$ is the corresponding pseudofunctor for which $D_G^b(\Per(X)) = \PC(D)$,
$D(X) = D^b(\Per(G \backslash X)).$
\end{proof}

We now recall the definitions of what it means to be a triangulated and truncated pseudofunctor as given in \cite{Monograph}. These describe the structure on a pseudofunctor sufficient to determine when the category $\PC(F)$ is triangulated (with triangulation induced by triangulations on the fibre categories $F(X)$) and admits $t$-structures which are given by $t$-structures on the fibre categories $F(X)$.
\begin{definition}[{\cite[{Definition 5.1.9}]{Monograph}}]\label{Defn: Recollection: triangulated pseuodofunctor}
	A pseudofunctor $F:\Cscr^{\op} \to \fCat$ is triangulated if for every object $X \in \Cscr_0$ the category $F(X)$ is triangulated and if for every morphism $f \in \Cscr_1$ the functor $F(f)$ is triangulated.
\end{definition}
\begin{definition}[{\cite[{Definition 5.1.18}]{Monograph}}]\label{Defn: Recollection: truncated pseudofunctor}
	Let $F:\Cscr^{\op} \to \fCat$ be a triangulated pseudofunctor. Assume each fibre category has $t$-structure with truncation functors $(\quot{\tau^{\leq 0}}{X},\quot{\tau^{\geq 0}}{X})$ such that these give rise to pseudofunctors $F^{\leq 0}:\Cscr^{\op} \to \fCat$ and $F^{\geq 0}:\Cscr^{\op} \to \fCat$ for which the functors $F^{\geq 0}(f)$ and $F^{\leq 0}(f)$ are $t$-exact functors which coincide on all common subcategories. Then if for all morphisms $f:X \to Y$ in $\Cscr$ there are natural isomorphisms
\[
\begin{tikzcd}
	F(Y) \ar[r, ""{name = U}]{}{F(f)} \ar[d, swap]{}{\quot{\tau^{\leq 0}}{Y}} & F(X) \ar[d]{}{\quot{\tau^{\leq 0}}{X}} \\
	F(Y)^{\leq 0} \ar[r, swap, ""{name = L}]{}{F^{\leq 0}(f)} & F(X)^{\leq 0} \ar[from = U, to = L, Rightarrow, shorten <= 4pt, shorten >= 4pt]{}{\theta_{f}^{\leq 0}}
\end{tikzcd}
\]
and
\[
\begin{tikzcd}
	F(Y) \ar[r, ""{name = U}]{}{F(f)} \ar[d, swap]{}{\quot{\tau^{\geq 0}}{Y}} & F(X) \ar[d]{}{\quot{\tau^{\geq 0}}{X}} \\
	F(Y)^{\geq 0} \ar[r, swap, ""{name = L}]{}{F^{\geq 0}(f)} & F(X)^{\geq 0} \ar[from = U, to = L, Rightarrow, shorten <= 4pt, shorten >= 4pt]{}{\theta_{f}^{\geq 0}}
\end{tikzcd}
\]
which satisfy the coherence diagrams presented below, we say that the pseudofunctor is truncated. The pasting diagram
\[
\begin{tikzcd}
	F(Z) \ar[rr, ""{name = UL}]{}{F(g)} \ar[d, swap]{}{\quot{\tau^{\leq 0}}{Z}} \ar[rrrr, bend left = 30, ""{name = U}]{}{F(g \circ f)} &  & F(Y) \ar[d]{}{\quot{\tau^{\leq 0}}{Y}} \ar[rr, ""{name = UR}]{}{F(f)} & & F(X) \ar[d]{}{\quot{\tau^{\leq 0}}{X}} \\
	F(Z)^{\leq 0} \ar[rr, swap, ""{name = LL}]{}{F^{\leq 0}(g)} \ar[rrrr, bend right = 30, swap, ""{name = L}]{}{F^{\leq 0}(g \circ f)} & & F(Y)^{\leq 0} \ar[rr, swap, ""{name = LR}]{}{F^{\leq 0}(f)} & & F(X)^{\leq 0} \ar[from = U, to = 1-3, Rightarrow, shorten <= 4pt, shorten <= 4pt]{}{\quot{\phi_{f,g}}{F}^{-1}} \ar[from = UR, to = LR, Rightarrow, shorten <= 4pt, shorten >= 4pt]{}{\theta_{f}^{\leq 0}} \ar[from = UL,, to = LL, Rightarrow, shorten <= 4pt, shorten >= 4pt]{}{\theta_g^{\leq 0}} \ar[from = 2-3, to = L, shorten <= 4pt, shorten >= 4pt, Rightarrow]{}{\quot{\phi_{f,g}}{F^{\leq 0}}}
\end{tikzcd}
\]
is equal to the $2$-cell
\[
\begin{tikzcd}
	F(Z) \ar[rr, ""{name = U}]{}{F(g \circ f)} \ar[d, swap]{}{\quot{\tau^{\leq 0}}{Z}} & & F(X) \ar[d]{}{\quot{\tau^{\leq 0}}{X}} \\
	F(Z)^{\leq 0} \ar[rr, swap, ""{name = L}]{}{F^{\leq 0}( g \circ f)} & & F(X)^{\leq 0} \ar[from = U, to = L, Rightarrow, shorten <= 4pt, shorten >= 4pt]{}{\theta_{g \circ f}^{\leq 0}}
\end{tikzcd}
\]
and dually the pasting diagram
\[
\begin{tikzcd}
	F(Z) \ar[rr, ""{name = UL}]{}{F(g)} \ar[d, swap]{}{\quot{\tau^{\geq 0}}{Z}} \ar[rrrr, bend left = 30, ""{name = U}]{}{F(g \circ f)} &  & F(Y) \ar[d]{}{\quot{\tau^{\geq 0}}{Y}} \ar[rr, ""{name = UR}]{}{F(f)} & & F(X) \ar[d]{}{\quot{\tau^{\geq 0}}{X}} \\
	F(Z)^{\geq 0} \ar[rr, swap, ""{name = LL}]{}{F^{\geq 0}(g)} \ar[rrrr, bend right = 30, swap, ""{name = L}]{}{F^{\geq 0}(g \circ f)} & & F(Y)^{\geq 0} \ar[rr, swap, ""{name = LR}]{}{F^{\geq 0}(f)} & & F(X)^{\geq 0} \ar[from = U, to = 1-3, Rightarrow, shorten <= 4pt, shorten <= 4pt]{}{\quot{\phi_{f,g}}{F}^{-1}} \ar[from = UR, to = LR, Rightarrow, shorten <= 4pt, shorten >= 4pt]{}{\theta_{f}^{\geq 0}} \ar[from = UL,, to = LL, Rightarrow, shorten <= 4pt, shorten >= 4pt]{}{\theta_g^{\geq 0}} \ar[from = 2-3, to = L, shorten <= 4pt, shorten >= 4pt, Rightarrow]{}{\quot{\phi_{f,g}}{F^{\geq 0}}}
\end{tikzcd}
\]
is equal to the $2$-cell:
\[
\begin{tikzcd}
	F(Z) \ar[rr, ""{name = U}]{}{F(g \circ f)} \ar[d, swap]{}{\quot{\tau^{\geq 0}}{Z}} & & F(X) \ar[d]{}{\quot{\tau^{\geq 0}}{X}} \\
	F(Z)^{\geq 0} \ar[rr, swap, ""{name = L}]{}{F^{\geq 0}( g \circ f)} & & F(X)^{\geq 0} \ar[from = U, to = L, Rightarrow, shorten <= 4pt, shorten >= 4pt]{}{\theta_{g \circ f}^{\geq 0}}
\end{tikzcd}
\]
\end{definition}

WWith these definitions and \cite{Monograph}, we now get that each flavour of $D_G^b(\Per(X))$ is a triangulated category with a $t$-structure equivalent to the category of equivariant perverse sheaves.
\begin{proposition}
Let $D_G^b(\Per(X))$ be any of the flavours of equivariant bounded derived category of perverse sheaves in Definition \ref{Defn: Section Recall: Pseudofunctor of derived perverse sheaves}. Then $D_G^b(\Per(X))$ is a triangulated category with $t$-structure whose heart satisfies
\[
D^b_G(\Per(X))^{\heartsuit} \simeq \Per_G(X)
\]
where $\Per_G(X)$ is the corresponding flavour of equivariant perverse sheaves.
\end{proposition}
\begin{proof}
That $D_G^b(\Per(X))$ is triangulated follows from \cite[Theorem 5.1.10]{Monograph} the fact that each category $D^b(\Per_G(G \backslash (\Gamma \times X)))$ is triangulated and each fibre functor ${}^{p}(G \backslash (\Gamma \times X))^{\ast}$ is triangulated exact. The $t$-structure on $D_G^b(\Per(X))$ is induced by \cite[Theorem 5.1.21]{Monograph}, the fact that each bounded derived category $D^b(\Per_G(G \backslash (\Gamma \times X)))$ has a standard $t$-structure, each fibre functor is exact for this $t$-structure, and the collection of truncation/inclusion adjunctions gives rise to a truncation structure for the corresponding pseudofunctor. Finally the Change of Heart Theorem (cf.\@ \cite[Theorem 5.1.24]{Monograph}) gives that
\begin{align*}
D_G^b(\Per(X))^{\heartsuit} &= \PC\left(D^b(\Per_G \backslash (-)))\right)^{\heartsuit} = \PC\left(D^b(\Per(G \backslash (-)))^{\heartsuit}\right) \\
&\simeq \PC(\Per(G \backslash (-))) \\
&\simeq \Per_G(X);
\end{align*}
note the last equivalence follows from techinques similar to those used in \cite[Corollary 5.1.28]{Monograph}.
\end{proof}

\section{Pseudocone Realization Functors}\label{Section: Pseudocone Realization Functors}
In this section we show how to construct realization functors $\PC(F) \to \PC(T)$ which are given in terms of local realization functors $F(X) \to T(X)$. Our strategy is to use these pseudoconical realizations to prove the equivariant version of Beilinson's Theorem; as such, we develop them in a relatively high level of generality. This section is, in essence, a pseudoconical version of \cite[{Appendix A}]{Beilinson} and gives a notion of pseudocone filtered categories and tools for descent filtrations.

\begin{definition}[{\cite[Definiiton A.1.a]{Beilinson}}]
Let $\Tscr$ be a triangulated category. We say that $\Tscr$ is filtered (or that $\Tscr$ is an $f$-category) if:
\begin{enumerate}
	\item $\Tscr$ has two strictly full subcategories $\Tscr^{f\geq 0}$ and $\Tscr^{f\leq 0}$.
	\item $\Tscr$ has a triangulated automorphism $s:\Tscr \to \Tscr$ (which we think of as a shift of filtration) together with a natural transformation $\alpha:\id_{\Tscr} \Rightarrow s$.
	\item If we define, for any $n \in \Z$,
	\[
	\Tscr^{f\leq n} := s^{n}\left(\Tscr^{f\leq 0}\right), \quad \Tscr^{f\geq n} := s^{n}\left(\Tscr^{f\geq 0}\right)
	\]
	then
	\[
	\bigcup_{n \in \Z} \Tscr^{f\geq n} = \Tscr = \bigcup_{n \in \Z} \Tscr^{f\leq n}.
	\]
	\item The identity $\alpha = s \ast \alpha \ast s^{-1}$ holds.
	\item For all $X, Y \in \Tscr_0$ if $X \in (\Tscr^{f\geq 1})_0$ and if $Y \in (\Tscr^{f\leq 0})_0$ then $\Tscr(X,Y) = 0$. Furthermore, $\alpha$ induces isomorphisms
	\[
	\Tscr(Y,s^{-1}X) \xrightarrow{\Tscr\left(Y,(\alpha \ast s^{-1})_X\right)} \Tscr(Y,X)
	\]
	and
	\[
	\Tscr(sY,X) \xrightarrow{\Tscr(\alpha_Y,X)} \Tscr(Y,X).
	\]
\end{enumerate}
We will call the information $(\Tscr^{f\leq 0}, \Tscr^{f\geq 0}, s, \alpha)$ the corresponding filtration on $\Tscr$.
\end{definition}
\begin{definition}[{\cite[{Definition A.1.b}]{Beilinson}}]
Let $\Cscr, \Tscr$ be filtered triangulated categories. We say that a functor $F:\Cscr \to \Tscr$ is filtered if $F$ is triangulated, $F \circ s_{\Cscr} = s_{\Tscr} \circ F$, $\alpha_{\Tscr} \ast F = F \ast \alpha_{\Cscr}$, and if $F$ preserves $\Cscr^{f\leq 0}$ and $\Cscr^{f\geq 0}$ in the sense that the diagrams
\[
\begin{tikzcd}
\Cscr^{f\leq 0} \ar[r]{}{\incl^{f\leq 0}_{\Cscr}} \ar[d, swap]{}{F|_{\Cscr^{f\leq 0}}} & \Cscr \ar[d]{}{F} \\
\Tscr^{f\leq 0} \ar[r, swap]{}{\incl_{\Tscr}^{f\leq 0}} & \Tscr
\end{tikzcd}
\]
and
\[
\begin{tikzcd}
	\Cscr^{f\geq 0} \ar[r]{}{\incl^{f\geq 0}_{\Cscr}} \ar[d, swap]{}{F|_{\Cscr^{f\geq 0}}} & \Cscr \ar[d]{}{F} \\
	\Tscr^{f\geq 0} \ar[r, swap]{}{\incl_{\Tscr}^{f\geq 0}} & \Tscr
\end{tikzcd}
\]
commute.
\end{definition}

Our first result in this section is that if we have a pseudofunctor $F$ which is coherently filtered, then the pseudocone category $\PC(F)$ is filtered.
\begin{proposition}\label{Prop: Section Pseudocone Realizations: PCF filtered}
Let $F:\Cscr^{\op} \to \fCat$ be a triangulated pseudofunctor for which:
\begin{itemize}
	\item For all $X \in \Cscr_0$ the category $F(X)$ is filtered with filtration given by the quadruple $(F(X)^{f\leq 0}, F(X)^{f\geq 0}, \quot{s}{X}, \quot{\alpha}{X})$.
	\item For all morphisms $f:X \to Y$ in $\Cscr$ the functor $F(f)$ is a filtered functor.
\end{itemize}
Then $\PC(F)$ is an $f$-category.
\end{proposition}
\begin{proof}
We define the strictly full subcategories $\PC(F)^{f\leq 0}$ and $\PC(F)^{f\geq 0}$ by taking the objects to be given by
\[
\left(\PC(F)^{f\leq 0}\right)_0 := \left\lbrace A \in \PC(F)_0 \; : \; \forall\,X \in \Cscr_0. \quot{A}{X} \in \left(F(X)^{f\leq 0}\right)_0 \right\rbrace
\]
and
\[
\left(\PC(F)^{f\geq 0}\right)_0 := \left\lbrace A \in \PC(F)_0 \; : \; \forall\,X \in \Cscr_0. \quot{A}{X} \in \left(F(X)^{f\geq 0}\right)_0 \right\rbrace
\]
and generating the morphisms appropriately in both cases. The assumptions regarding the fibre functors $F(f)$ and the shift functors $\quot{s}{X}$ imply that if for any morphism $f:X \to Y$ in $\Cscr$ we set $\quot{s}{f}$ to be the commutativity witness
\[
\begin{tikzcd}
F(Y) \ar[d, swap]{}{F(f)} \ar[r, ""{name = U}]{}{\quot{s}{Y}} & F(Y) \ar[d]{}{F(f)} \\
F(X) \ar[r, swap, ""{name = D}]{}{\quot{s}{X}} & F(X) \ar[from = U, to = D, Rightarrow, shorten <= 4pt, shorten >= 4pt]{}{\quot{s}{f}}
\end{tikzcd}
\]
then the collection $s = (\quot{s}{X}, \quot{s}{f})_{X \in \Cscr_0, f \in \Cscr_1}$ determines a pseudonatural transformation $\ul{s}:F \Rightarrow F$. Consequently, setting $s := \PC(\ul{s}):\PC(F) \to \PC(F)$, where $\PC$ is the (strict) pseudoconification $2$-functor of \cite[{Lemma 4.1.13}]{Monograph}, defines our shift automorphism. Similarly, we define our remaining component is done by first noticing that if we define the collection
\[
\ul{\alpha} := \left\lbrace \quot{\alpha}{X}:\id_{F(X)} \Rightarrow \quot{s}{X} \; | \; X \in \Cscr_0 \right\rbrace
\]
then for any map $f:X \to Y$ in $\Cscr$ the diagram of functors and natural transformations
\[
\begin{tikzcd}
F(f) \ar[rr]{}{F(f) \ast \quot{\alpha}{Y}} \ar[d, equals] & & F(f) \circ \quot{s}{Y} \ar[d]{}{\quot{s}{f}} \\
F(f) \ar[rr, swap]{}{\quot{\alpha}{Y} \ast F(f)} & & \quot{s}{X} \circ F(f)
\end{tikzcd}
\]
commutes exactly because $F(f)$ is an $f$-functor. Consequently $\ul{\alpha}$ determines a modification:
\[
\begin{tikzcd}
	\Cscr^{\op} \ar[rrr, bend left = 40, ""{name = Up}]{}{F} \ar[rrr, bend right = 40, swap, ""{name = Down}]{}{F} & & & \fCat \ar[from = Up, to = Down, Rightarrow, shorten <= 4pt, shorten >= 4pt, bend left = 40, ""{name = Left}]{}{\ul{s}} \ar[from = Up, to = Down, Rightarrow, shorten <= 4pt, shorten >= 4pt, bend right = 40, swap, ""{name = Right}]{}{\id_{F}} \ar[from = Right, to = Left, symbol = \underset{\ul{\alpha}}{\Rrightarrow}]
\end{tikzcd}
\]
We thus define $\alpha:\id_{\PC(F)} \to s$ by $\alpha := \PC(\ul{\alpha})$. That this satisfies the required axioms of a filtered category then follows from the fact that for all $X \in \Cscr_0$ the corresponding identity holds. For example, that $\alpha = s \ast \alpha \ast s$ holds follows from the fact that for all $X \in \Cscr_0$, $\quot{\alpha}{X} = \quot{s}{X} \ast \quot{\alpha}{X} \ast \quot{s^{-1}}{X}.$
\end{proof}
\begin{definition}\label{Defn: Section Pseudocone Realizations: Filtered Pseudofunctor}
Let $F:\Cscr^{\op} \to \fCat$ be a triangulated pseudofunctor. If each category $F(X)$ is filtered for all $X \in \Cscr_0$ and if $F(f)$ is an $f$-functor for all $f \in \Cscr_1$ then we say that $F$ is a triangulated filtered pseudofunctor.
\end{definition}
\begin{remark}\label{Remark: Section Pseudocone Realizations: The leq and geq pseudocones} 
If $F$ is a filtered triangulated pseudofunctor then the assumptions imply that there are pseudofunctors $F^{\fil\leq 0},$ $F^{\fil\geq 0},$ and $F^{\fil\leq 0} \cap F^{\fil\geq 0}$ which are defined by
\[
F^{\fil\geq 0}(X) := F(X)^{\fil\geq 0}, \quad F^{\fil\geq 0}(\varphi:X \to Y) := F(\varphi)|_{F(X)^{\fil\geq 0}},
\]
and similarly for each of the other cases. These pseudofunctors have induced inclusion pseudonatural transformations which fit into a commuting diagram of triangulated pseudofunctors
\[
\begin{tikzcd}
F^{\fil\geq 0} \cap F^{\fil\leq 0} \ar[drr]{}[description]{\incl_{\fil\geq 0, \fil\leq 0}} \ar[rr]{}{\incl_{\fil\geq 0, \fil\leq 0}^{\fil\leq 0}} \ar[d, swap]{}{\incl_{\fil\geq 0, \fil\leq 0}^{\fil\geq 0}} & & F^{\fil\leq 0} \ar[d]{}{\incl_{\fil\leq 0}} \\
F^{\fil\geq 0} \ar[rr, swap]{}{\incl_{\fil\geq 0}} & & F
\end{tikzcd}
\]
and satisfy
\[
\PC\left(F^{\fil\leq 0}\right) = \PC(F)^{\fil\leq 0}, \quad \PC\left(F^{\fil\geq 0}\right) = \PC(F)^{\fil\geq 0},
\]
and
\[
\PC\left(F^{\fil\geq 0} \cap F^{\fil\leq 0}\right) = \PC(F)^{\fil\geq 0} \cap \PC(F)^{\fil\leq 0}.
\]
Note also that each of the inclusion pseudonatural transformations are triangulated in the sense that each functor $\quot{\incl}{X}$ is triangulated for all $X \in \Cscr_0$. These then consequently fit into a commuting diagram of triangulated categories
\[
\begin{tikzcd}
	\PC(F)^{\fil\geq 0} \cap \PC(F)^{\fil\leq 0} \ar[drr]{}[description]{\incl_{\fil\geq 0, \fil\leq 0}} \ar[rr]{}{\incl_{\fil\geq 0, \fil\leq 0}^{\fil\leq 0}} \ar[d, swap]{}{\incl_{\fil\geq 0, \fil\leq 0}^{\fil\geq 0}} & & \PC(F)^{\fil\leq 0} \ar[d]{}{\incl_{\fil\leq 0}} \\
	\PC(F)^{\fil\geq 0} \ar[rr, swap]{}{\incl_{\fil\geq 0}} & & \PC(F)
\end{tikzcd}
\]
with, of course, $\PC(F)$ a filtered triangulated category.
\end{remark}

We now recall what it means to have a filtered category of a triangulated category.
\begin{definition}[{\cite[{Definition A.1.c}]{Beilinson}}]
Let $\Cscr$ be a triangulated category. A filtered category over $\Cscr$ (also known as an $f$-category over $\Cscr$) is a filtered category $\Tscr$ together with a triangulated equivalence of triangulated categories
\[
i:\Cscr \to \Tscr^{\fil\geq 0} \cap \Tscr^{\fil\leq 0}.
\]
\end{definition}
	
We now show that if we have a filtered triangulated pseudofunctor $F$ and a triangulated pseudofunctor $T$ for which each fibre category $F(X)$ is an $f$-category over $T(X)$ for which the requisite equivalences vary pseudonaturally, then $\PC(F)$ is an $f$-category over $\PC(T)$.
	
\begin{proposition}\label{Prop: Section Pseudocone Realizations: When PCF is filtered over PCT}
Let $F:\Cscr^{\op} \to \fCat$ be a filtered triangulated pseudofunctor and let $T:\Cscr^{\op} \to \fCat$ be a triangulated. If:
\begin{enumerate}
	\item For each $X \in \Cscr_0$ we have triangulated equivalences of categories
	\[
	\quot{i}{X}:T(X) \xrightarrow{\simeq} F(X)^{\fil\geq 0} \cap F(X)^{\fil\leq 0};
	\]
	\item The equivalences $\quot{i}{X}$ vary pseudonaturally in $\Cscr^{\op}$;
\end{enumerate}
then $\PC(F)$ is an $f$-category over $\PC(T)$.
\end{proposition}
\begin{proof}
The assumptions of the proposition imply that there is a pseudonatural equivalence
\[
i:T \xRightarrow{\simeq} F^{\fil\geq 0} \cap F^{\fil\leq 0}
\]
which extends to a triangulated equivalence of categories
\[
\PC(i):\PC(T) \xrightarrow{\simeq} \PC\left(F^{\fil\geq 0} \cap F^{\fil\leq 0}\right) = \PC(F)^{\fil\geq 0} \cap \PC(F)^{\fil\leq 0};
\]
that $\PC(i)$ is triangulated follows from the fact that each functor $\quot{i}{X}$ is triangulated and the triangulations on $\PC(T)$ and $\PC(F)^{\fil\geq 0} \cap \PC(F)^{\fil\leq 0}$ are locally induced by the triangulations on the categories $T(X)$ and $F(X)^{\fil\geq 0} \cap F(X)^{\fil\leq 0}$ by \cite[{Theorem 5.1.10}]{Monograph}.
\end{proof}	

This motivates the definition of what we call an $f$-pseudofunctor over another triangulated pseudofunctor. 
\begin{definition}\label{Defn: Section Pseudocone Realizations: fpseudofunctor}
Let $F:\Cscr^{\op} \to \fCat$ be a filtered triangulated pseudofunctor and let $T:\Cscr^{\op} \to \fCat$ be a triangulated. If:
\begin{enumerate}
	\item For each $X \in \Cscr_0$ we have triangulated equivalences of categories
	\[
	\quot{i}{X}:T(X) \xrightarrow{\simeq} F(X)^{\fil\geq 0} \cap F(X)^{\fil\leq 0};
	\]
	\item The equivalences $\quot{i}{X}$ vary pseudonaturally in $\Cscr^{\op}$;
\end{enumerate}
then we say that $F$ is a filtered triangulated pseudofunctor (or an $f$-pseudofunctor) over $T$.
\end{definition}
In this new terminology, Proposition \ref{Prop: Section Pseudocone Realizations: When PCF is filtered over PCT} can be restated as saying that if $F$ is an $f$-pseudofunctor over a triangulated pseudofunctor $T$, then $\PC(F)$ is an $f$-category over the triangulated category $\PC(T)$.

Continuing to follow \cite[{Appendix A}]{Beilinson}, we now assume that we have an $f$-pseudofunctor $F:\Cscr^{\op} \to \fCat$ over a triangulated pseudofunctor $T:\Cscr^{\op} \to \fCat$. Because by Proposition \ref{Prop: Section Pseudocone Realizations: When PCF is filtered over PCT} $\PC(F)$ is an $f$-category over $\PC(T)$, for every $n \in \Z$ there are adjunctions
\[
\begin{tikzcd}
\PC(F)  \ar[rr, swap, bend right = 20, ""{name = D}]{}{\sigma_{\fil\geq n}} & & \PC(F)^{\fil\geq n} \ar[ll, swap, bend right = 20, ""{name = U}]{}{\incl_{\fil\geq n}} \ar[from = U, to = D, symbol = \dashv]
\end{tikzcd}
\]
and
\[
\begin{tikzcd}
	\PC(F)^{\fil\leq n}  \ar[rr, swap, bend right = 20, ""{name = D}]{}{\incl_{\fil\leq n}} & & \PC(F) \ar[ll, swap, bend right = 20, ""{name = U}]{}{\sigma_{\fil\leq n}} \ar[from = U, to = D, symbol = \dashv]
\end{tikzcd}
\]	
for which each of the functors $\sigma_{\fil\geq n}$ and $\sigma_{\fil\leq n}$ are both triangulated by \cite[{Proposition A.3.i}]{Beilinson}. Furthermore, by construction and the fact that $F$ is an $f$-pseudofunctor over $T$, we have local adjunctions
\[
\begin{tikzcd}
	F(X)  \ar[rrr, swap, bend right = 20, ""{name = D}]{}{\quot{\sigma_{\fil\geq n}}{X}} & & & F(X)^{\fil\geq n} \ar[lll, swap, bend right = 20, ""{name = U}]{}{\quot{\incl_{\fil\geq n}}{X}} \ar[from = U, to = D, symbol = \dashv]
\end{tikzcd}
\]
and:
\[
\begin{tikzcd}
	F(X)^{\fil\leq n}  \ar[rrr, swap, bend right = 20, ""{name = D}]{}{\quot{\incl_{\fil\leq n}}{X}} & & & F(X) \ar[lll, swap, bend right = 20, ""{name = U}]{}{\quot{\sigma_{\fil\leq n}}{X}} \ar[from = U, to = D, symbol = \dashv]
\end{tikzcd}
\]
Furthermore, because the functors $\quot{\incl_{\fil\geq n}}{X}$ and $\quot{\incl_{\fil\leq n}}{X}$ are the object functors of pseudonatural transformations $\ul{\incl}_{\fil\geq n}:F^{\fil\geq n} \Rightarrow F$ and $\ul{\incl}_{\fil\leq n}:F^{\fil\leq n} \Rightarrow F$, it follows that the $\quot{\sigma_{\fil\geq n}}{X}$ and $\quot{\sigma_{\fil\geq n}}{X}$ are the object functors of pseudonatural transformations
\[
\ul{\sigma}_{\fil\geq n}:F \Rightarrow F^{\fil\geq n}, \quad \ul{\sigma}_{\fil\leq n}:F \Rightarrow F^{\fil\leq n}.
\]
It then follows from \cite[{Theorem 4.1.17}]{Monograph} that we have adjunctions 
\[
\ul{\incl}_{\fil\geq n} \dashv \ul{\sigma}_{\fil\geq n}:F^{\fil\geq n} \Rightarrow F
\] 
and 
\[
\ul{\sigma}_{\fil\leq n} \dashv \ul{\incl}_{\fil\leq n}:F \Rightarrow F^{\fil\leq n}
\] 
in the $2$-category $\Bicat(\Cscr^{\op},\fCat)$. Consequently we deduce that the adjoints 
\[
\incl_{\fil\geq n} \dashv \sigma_{\fil\geq n}:PC(F)^{\fil\geq n} \to \PC(F)
\]
and
\[
\sigma_{\fil\leq n} \dashv \incl_{\fil\leq n}:\PC(F) \to \PC(F)^{\fil\leq n}
\]
are induced as the psuedoconification of the corresponding adjoints in $\Bicat(\Cscr^{\op},\fCat)$. This allows us to in turn deduce the complete pseudocone analogue of \cite[{Proposition A.3.i}]{Beilinson}.
\begin{proposition}\label{Prop: Section Pseduocone realizations: The theta nisos}
Let $F:\Cscr^{\op} \to \fCat$ be a filtered triangulated pseudofunctor. Then for every $n \in \Z$ there are adjunctions
\[
\incl_{\fil\geq n} \dashv \sigma_{\fil\geq n}:PC(F)^{\fil\geq n} \to \PC(F)
\]
and
\[
\sigma_{\fil\leq n} \dashv \incl_{\fil\leq n}:\PC(F) \to \PC(F)^{\fil\leq n}
\]
for which:
\begin{enumerate}
	\item Each of the functors $\sigma_{\fil\leq n}$ and $\sigma_{\fil\geq n}$ are triangulated.
	\item The functors $\sigma_{\fil\leq n}$ and $\sigma_{\fil\geq n}$ preserve the categories $\PC(F)^{\fil\leq n}$ and $\PC(F)^{\fil\geq n}$, respectively.
	\item For all $m, n \in \Z$ there is a natural isomorphism $\theta_{\fil\leq n}^{\fil\geq m}$ which makes the diagram
	\[
	\begin{tikzcd} 
	\incl_{\fil\geq m} \circ \sigma_{\fil\geq m} \ar[ddd, swap]{}{\eta^{\fil\leq n} \ast (\incl_{\fil\geq m} \circ \sigma_{\fil\geq m})} \ar[r]{}{\epsilon^{\fil\geq m}} & \id_{\PC(F)} \ar[d]{}{\eta^{\fil\leq n}} 	\\
	 & \incl_{\fil\leq n} \circ \sigma_{\fil\leq n} \\
	 & \incl_{\fil\geq m} \circ \sigma_{\fil\geq m} \circ \incl_{\fil\leq n} \circ \sigma_{\fil\leq n} \ar[u, swap]{}{\epsilon^{\fil\geq m} \ast (\incl_{\fil\leq n} \circ \sigma_{\fil\leq n})} \\
	 \incl_{\fil\leq n} \circ \sigma_{\fil\leq n} \circ \incl_{\fil\geq m} \circ \sigma_{\fil\geq m} \ar[ur, swap]{}{\theta_{\fil\leq n}^{\fil\geq m}}
	\end{tikzcd}
	\]
	of functors and natural transformations commute.
\end{enumerate}
\end{proposition}
\begin{proof}
(1): The triangulation of $\sigma_{\fil\leq n}$ and $\sigma_{\fil\geq n}$ are immediate from \cite[{Proposition A.3.i}]{Beilinson} or, alternatively, the fact that all the object functors of $\sigma_{\fil\leq n} = \PC(\ul{\sigma}_{\fil\leq n})$ and $\sigma_{\fil\geq n} = \PC(\ul{\sigma}_{\fil\geq n})$ are triangulated by \cite[{Proposition A.3.i}]{Beilinson}.

(2): This follows from the fact that for every $X \in \Cscr_0$, by \cite[{Proposition A.3.i}]{Beilinson} the functors $\quot{\sigma_{\fil\leq n}}{X}$ and $\quot{\sigma_{\fil\geq n}}{X}$ preserve $F(X)^{\fil\leq n}$ and $F(X)^{\fil\geq n}$.

(3): By \cite[{Proposition A.3.i}]{Beilinson} for every $X \in \Cscr_0$ there is a unique natural isomorphism $\quot{\theta_{\fil\leq n}^{\fil\geq m}}{X}$ making the diagram
\[
\begin{tikzcd} 
	\quot{\incl}{X}_{\fil\geq m} \circ \quot{\sigma}{X}_{\fil\geq m} \ar[ddd, swap]{}{\quot{\eta}{X}^{\fil\leq n} \ast (\quot{\incl}{X}_{\fil\geq m} \circ \quot{\sigma}{X}_{\fil\geq m})} \ar[r]{}{\quot{\epsilon}{X}^{\fil\geq m}} & \id_{\PC(F)} \ar[d]{}{\quot{\eta}{X}^{\fil\leq n}} 	\\
	& \quot{\incl}{X}_{\fil\leq n} \circ \quot{\sigma}{X}_{\fil\leq n} \\
	& \\
	\quot{\incl}{X}_{\fil\leq n} \circ \quot{\sigma}{X}_{\fil\leq n} \circ \quot{\incl}{X}_{\fil\geq m} \circ \quot{\sigma}{X}_{\fil\geq m} \ar[d, swap]{}{\quot{\theta}{X}_{\fil\leq n}^{\fil\geq m}} \\
	 \quot{\incl}{X}_{\fil\geq m} \circ \quot{\sigma}{X}_{\fil\geq m} \circ \quot{\incl}{X}_{\fil\leq n} \circ \quot{\sigma}{X}_{\fil\leq n} \ar[uuur, bend right = 20, swap]{}{\quot{\epsilon}{X}^{\fil\geq m} \ast (\quot{\incl}{X}_{\fil\leq n} \circ \quot{\sigma}{X}_{\fil\leq n})}
\end{tikzcd}
\]
commute; consequently, if we can show that the collection 
\[
\Theta_{\fil\leq n}^{\fil\geq m} = \left\lbrace \quot{\theta_{\fil\leq n}^{\fil\geq m}}{X} \; | \; X \in \Cscr_0 \right\rbrace
\]
is a modification then setting $\theta_{\fil\leq n}^{\fil\geq m} := \PC(\Theta_{\fil\leq n}^{\fil\geq m})$ completes the proof. To this end note that showing for any $f:X \to Y$ in $\Cscr$, $\Theta_{\fil\leq n}^{\fil\geq n}$ being a modification is equivalent to proving that the equation
\begin{align*}
 &\left(\quot{{\theta}_{\fil\leq n}^{\fil\geq m}}{X} \ast F(f)\right) \circ \quot{\left(\incl_{\fil\leq n} \circ \sigma_{\fil\leq n} \circ \incl_{\fil\geq m} \circ \sigma_{\fil\geq m}\right)}{f} \\
 &= \quot{\left(\incl_{\fil\geq m} \circ \sigma_{\fil\geq m} \circ \incl_{\fil\leq n} \circ \sigma_{\fil\leq n}\right)}{f} \circ \left(F(f) \ast \quot{\theta_{\fil\leq n}^{\fil\geq m}}{Y}\right)
\end{align*}
commutes follows from the naturality and uniqueness of $\quot{\theta_{\fil\leq n}^{\fil\geq m}}{Y}$ and $\quot{\theta_{\fil\leq n}^{\fil\geq m}}{X}$. 
\end{proof}
These natural transformations and functors allow us to define a pseudonatural transformation $F \Rightarrow T$ which object-locally gives the $n$-th graded component of $F(X)$ as a filtered category over $T(X)$.
\begin{lemma}\label{Lemma: Section Pseudocone Realization: Graded PC functor}
Let $F:\Cscr^{\op} \to \fCat$ be a filtered triangulated pseudofunctor over a triangulated pseudofunctor $T:\Cscr^{\op} \to \fCat$ and let $\ul{j}:F^{\fil\geq 0} \cap F^{\fil\leq 0} \Rightarrow T$ be an inverse pseudonatural equivalence to $\ul{i}:T \Rightarrow F^{\fil\geq 0} \cap F^{\fil\leq 0}$. For any $n \in \Z$ there is an $n$-th graded component pseudonatural transformation:
\[
\begin{tikzcd}
\Cscr^{\op} \ar[rr, bend left = 20, ""{name = U}]{}{F} \ar[rr, bend right = 20, swap, ""{name = D}]{}{T} & & \fCat \ar[from = U, to = D, Rightarrow, shorten <= 4pt, shorten >= 4pt]{}{\ul{\gr}_f^n}
\end{tikzcd}
\]
In particular this gives rise to an $n$-th graded component functor $\gr_{\fil}^{n}:\PC(F) \to \PC(T)$.
\end{lemma}
\begin{proof}
We define $\ul{\gr}_f^n$ by
\[
\ul{\gr}_f^n := \ul{j} \circ \ul{s}^{-n} \circ \ul{\incl}_{\fil\leq n} \circ \ul{\sigma}_{\fil\leq n} \circ \ul{\incl}_{\fil\geq n} \circ \ul{\sigma}_{\fil\geq n}.
\]
That this is justified as the $n$-th graded component pseudonatural transformation follows both from the fact that for all $X \in \Cscr_0$, the $X$-local component is given by the functor
\[
\quot{\ul{\gr}_f^n}{X} = \quot{j}{X} \circ \quot{{s}^{-n}}{X} \circ \quot{{\incl}_{\fil\leq n}}{X} \circ \quot{{\sigma}_{\fil\leq n}}{X} \circ \quot{{\incl}_{\fil\geq n}}{X} \circ \ul{\sigma}_{\fil\geq n},
\]
which is the $n$-th graded functor $F(X) \to T(X)$ described in \cite[{Proposition A.3.i}]{Beilinson}. The final statement of the lemma follows by taking $\gr_{\fil}^n := \PC(\gr_{\fil}^n)$.
\end{proof}
\begin{remark}\label{Remark: The Cap sigma notation}
As we have seen, it is frequently important to use the functors $\sigma_{\fil \leq n}$ and $\sigma_{\fil \geq n}$ followed by composition with the inclusions $\incl_{\fil \leq n}$ and $\incl_{\fil \geq n}$. Because of this, we let $\Sigma_{\fil \geq k} := \incl_{\fil \geq k} \circ \sigma_{\fil \geq k}$ and $\Sigma_{\fil \leq k} := \incl_{\fil \leq k} \circ \sigma_{\fil \geq k}$ for all $k \in \Z$ in order to save space while also being categorically rigorous. Note that this also has the benefit of allowing us to rewrite $\gr_{\fil}^n$ as $j \circ s^{-n} \circ \Sigma_{\fil\leq n} \circ \Sigma_{\fil\geq n}.$
\end{remark}

We now give a pseudoconical version of \cite[{Proposition A.3.ii}]{Beilinson}; however, we don't \emph{just} want to know that there is a distinguished triangle
\[
\begin{tikzcd}
\left(\Sigma_{\fil\geq 1}\right)(A) \ar[r]{}{\epsilon^{\fil\geq 1}_A} & A \ar[r]{}{\eta_{A}^{\fil\leq 0}} & \left(\Sigma_{\fil\leq 0}\right)(A) \ar[d]{}{d} \\
 & & \left(\Sigma_{\fil\geq 1}\right)(A)[1]
\end{tikzcd}
\]
in $\PC(F)$ for any object $A$. Instead, we want to know that this distinguished triangle comes from a family of distinguished triangles in each category $F(X)$.

\begin{proposition}\label{Prop: Section Pseudocone Realizations: Dist triangle pseudocone version}
Let $F:\Cscr^{\op} \to \fCat$ be a filtered triangulated pseudofunctor over a triangulated pseudofunctor $T:\Cscr^{\op} \to \fCat$. If $A$ is any object in $\PC(F)_0$ then there is a unique morphism $d:(\Sigma_{\fil\leq 0})(A) \to (\Sigma_{\fil\geq 1})(A)[1]$ for which there is a distinguished triangle:
\[
\begin{tikzcd}
	(\Sigma_{\fil\geq 1})(A) \ar[r]{}{\epsilon^{\fil\geq 1}_{A}} & A \ar[r]{}{\eta_{A}^{\fil\leq 0}} & (\Sigma_{\fil\leq 0})(A) \ar[d]{}{d} \\
	& & (\Sigma_{\fil\geq 1})(A)[1]
\end{tikzcd}
\]
Furthermore, for any $X \in \Cscr_0$ there are distinguished triangles
\[
\begin{tikzcd}
\left(\quot{\Sigma_{\fil\geq 1}}{X}\right)\left(\quot{A}{X}\right) \ar[r]{}{\quot{\epsilon^{\fil\geq 1}}{X}} & \quot{A}{X} \ar[r]{}{\quot{\eta^{\fil\leq 0}}{X}} & \left(\quot{\Sigma_{\fil\leq 0}}{X}\right)\left(\quot{A}{X}\right) \ar[d]{}{\quot{d}{X}} \\
 & & \left(\quot{\Sigma_{\fil\geq 1}}{X}\right)\left(\quot{A}{X}\right)[1]
\end{tikzcd}
\] 
and if $f:X \to Y$ is any morphism in $\Cscr$ the diagram
\[
\begin{tikzcd}
\left(\quot{\Sigma_{\fil\geq 1}}{X}\right)\left(F(f)(\quot{A}{Y})\right) \ar[d, swap]{}{\quot{\epsilon^{\fil\geq 1}_{F(f)A}}{X}} \ar[r]{}{} & \left(\quot{\Sigma_{\fil\geq 1}}{X}\right)\left(\quot{A}{X}\right) \ar[d]{}{\quot{\epsilon_A^{\fil\geq 1}}{X}} \\
\quot{A}{Y} \ar[d, swap]{}{\quot{\eta_{F(f)A}^{\fil\leq 0}}{X}} \ar[r]{}{\tau_f^A} & \quot{A}{X} \ar[d]{}{\quot{\eta^{\fil\leq 0}}{X}} \\
\left(\quot{\Sigma_{\fil\leq 0}}{X}\right)\left(F(f)\quot{A}{Y}\right) \ar[d, swap]{}{F(f)\left(\quot{d}{Y}\right)} \ar[r] & \left(\quot{\Sigma_{\fil\leq 0}}{X}\right)\left(\quot{A}{X}\right) \ar[d]{}{\quot{d}{X}} \\
\left(\quot{\Sigma_{\fil\geq 1}}{X}\right)\left(F(f)(\quot{A}{Y})\right) \ar[r]{}{} & \left(\quot{\Sigma_{\fil\geq 1}}{X}\right)\left(\quot{A}{X}\right)[1]
\end{tikzcd}
\]
commutes with every horizontal map the corresponding functor applied to $\tau_f^A$ or $\tau_f^{A[1]}$.
\end{proposition}
\begin{proof}
Applying \cite[{Proposition A.3.ii}]{Beilinson} gives a unique morphism $d:(\Sigma_{\fil\leq 0})(A) \to (\Sigma_{\fil\geq 1})(A)[1]$ for which there is a distinguished triangle
\[
\begin{tikzcd}
(\Sigma_{\fil\geq 1})(A) \ar[r]{}{\epsilon^{\fil\geq 1}_{A}} & A \ar[r]{}{\eta_{A}^{\fil\leq 0}} & (\Sigma_{\fil\leq 0})(A) \ar[d]{}{d} \\
 & & (\Sigma_{\fil\geq 1})(A)[1]
\end{tikzcd}
\]
in $\PC(F)$. By the construction of the triangulation on $\PC(F)$ given in \cite[{Theorem 5.1.10}]{Monograph} determined by saying a triangle is distinguished if and only if the corresponding object-local triangles are distinguished, we see that for every object $X$ of $\Cscr$ there is then a distinguished triangle:
\[
\begin{tikzcd}
	\left(\quot{\Sigma_{\fil\geq 1}}{X}\right)\left(\quot{A}{X}\right) \ar[r]{}{\quot{\epsilon^{\fil\geq 1}}{X}} & \quot{A}{X} \ar[r]{}{\quot{\eta^{\fil\leq 0}}{X}} & \left(\quot{\Sigma_{\fil\leq 0}}{X}\right)\left(\quot{A}{X}\right) \ar[d]{}{\quot{d}{X}} \\
	& & \left(\quot{\Sigma_{\fil\geq 1}}{X}\right)\left(\quot{A}{X}\right)[1]
\end{tikzcd}
\]
Note that the map $\quot{d}{X}$ is then the $X$-local version of the unique morphism of \cite[{Proposition A.3.ii}]{Beilinson}. Now because $A$ is an object of $\PC(F)$, for any morphism $f:X \to Y$ in $\Cscr$ the diagram
\[
\begin{tikzcd}
	F(f)\left(\left(\quot{\Sigma_{\fil\geq 1}}{Y}\right)\left(\quot{A}{Y}\right)\right) \ar[d, swap]{}{F(f)\left(\quot{\epsilon^{\fil\geq 1}_{A}}{Y}\right)} \ar[r]{}{} & \left(\quot{\Sigma_{\fil\geq 1}}{X}\right)\left(\quot{A}{X}\right) \ar[d]{}{\quot{\epsilon_A^{\fil\geq 1}}{X}} \\
	\quot{A}{Y} \ar[d, swap]{}{F(f)\left(\quot{\eta_{A}^{\fil\leq 0}}{Y}\right)} \ar[r]{}{\tau_f^A} & \quot{A}{X} \ar[d]{}{\quot{\eta^{\fil\leq 0}}{X}} \\
	F(f)\left(\quot{\Sigma_{\fil\leq 0}}{Y}\right)\left(\quot{A}{Y}\right) \ar[d, swap]{}{\quot{d_{F(f)A}}{X}} \ar[r] & \left(\quot{\Sigma_{\fil\leq 0}}{X}\right)\left(\quot{A}{X}\right) \ar[d]{}{\quot{d}{X}} \\
	F(f)\left(\left(\quot{\Sigma_{\fil\geq 1}}{Y}\right)\left((\quot{A}{Y})\right)\right) \ar[r]{}{} & \left(\quot{\Sigma_{\fil\geq 1}}{X}\right)\left(\quot{A}{X}\right)[1]
\end{tikzcd}
\]
commutes with each horizontal arrow a transition isomorphism. Finally the commutativity of the last stated diagram in the lemma follows by using pseudonatural transformation commutativity witness isomorphisms.
\end{proof}

We now want to build a pseudonatural transformation
\[
\begin{tikzcd}
	\Cscr^{\op} \ar[rr, bend left = 20, ""{name = U}]{}{F} \ar[rr, bend right = 20, swap, ""{name = D}]{}{T} & & \fCat \ar[from = U, to = D, Rightarrow, shorten <= 4pt, shorten >= 4pt]{}{\ul{\omega}}
\end{tikzcd}
\]
which interacts well with each inclusion functor, as this transformation is crucial to construction of the pseudocone realization functor in the same way that the corresponding functor $\omega$ allows us to build the realization functor in \cite{Beilinson}. For this we recall the construction and description of the $\omega$ functors in the conventional case.
\begin{proposition}[{\cite[{Proposition A.3.iii}]{Beilinson}}]\label{Prop: Section Pseudocone Realization: Omega functors in Beilinson}
Let $\Cscr$ be a filtered triangulated category over a triangulated category $\Ascr$. There is then a functor $\omega:\Cscr \to \Ascr$ such that:
\begin{enumerate}
	\item There is an adjunction:
	\[
	\begin{tikzcd}
	\Cscr^{\fil\leq 0} \ar[rrr, bend right = 20, swap, ""{name = D}]{}{\omega \circ \incl_{\fil\leq 0}} & & & \Ascr \ar[lll, bend right = 20, swap, ""{name = U}]{}{\incl_{\fil\geq 0, \fil\leq 0}^{\fil\leq 0} \circ i} \ar[from = U, to = D, symbol = \dashv]
	\end{tikzcd}
	\] 
	\item There is an adjunction:
		\[
	\begin{tikzcd}
		\Ascr \ar[rrr, bend right = 20, swap, ""{name = D}]{}{\incl_{\fil\geq 0, \fil\leq 0}^{\fil\geq 0} \circ i} & & & \Cscr^{\fil\geq 0} \ar[lll, bend right = 20, swap, ""{name = U}]{}{\omega \circ \incl_{\fil\geq 0}} \ar[from = U, to = D, symbol = \dashv]
	\end{tikzcd}
	\] 
	\item For any object $X$ of $\Ascr$, the morphism $\omega(\alpha_{iX}):\omega(iX) \to \omega(s(iX))$ is an isomorphism.
	\item For all objects $X$ of $\Cscr^{\fil\leq 0}$ and all objects $Y$ of $F^{\fil\geq 0}$, the map
	\[
	\Cscr(A,B) \to \Ascr(\omega(A), \omega(B))
	\]
	given by $\varphi \mapsto \omega(\varphi)$ is an isomorphism.
\end{enumerate}
\end{proposition}

Because of this proposition, to construct the pseudonatural version of $\omega$ (and hence the pseudoconical version by applying the $2$-functor $\PC$) we need to prove that these functors $\omega$ vary when in the presence of appropriate structure. With this in mind, assume that we have a filtered triangulated pseudofunctor $F:\Cscr^{\op} \to \fCat$ over a triangulated pseudofunctor $T:\Cscr^{\op} \to \fCat$. Asking that the $\omega$ vary suitably asks that if we write $\quot{\omega}{Z}:F(Z) \to T(Z)$ for each object $Z$ in $\Cscr$ as the functor of Proposition \ref{Prop: Section Pseudocone Realization: Omega functors in Beilinson}, then for any morphism $f:X \to Y$ in $\Cscr$, there be a natural isomorphism $\quot{\omega}{f}$ fitting into the $2$-cell
\[
\begin{tikzcd}
F(Y) \ar[r, ""{name = U}]{}{F(f)} \ar[d, swap]{}{\quot{\omega}{Y}} & F(X) \ar[d]{}{\quot{\omega}{X}} \\
T(Y) \ar[r, swap, ""{name = D}]{}{T(f)} & T(X) \ar[from = U, to = D, Rightarrow, shorten <= 4pt, shorten >= 4pt]{}{\quot{\omega}{f}}
\end{tikzcd}
\]
which varies pseudonaturally in $\Cscr^{\op}$. If they all exist, ensuring the adjoints are defined so as to be pseudonatural is no major issue; adjoints are unique up to unique isomorphism, so we can always redefine them if necessary to satisfy the coherence requirements. As such we simply need to ensure their existence. For this it is worth noting that since the functors $\omega$ are uniquely determined by either properties (1) and (3) or (2) and (3) in Proposition \ref{Prop: Section Pseudocone Realization: Omega functors in Beilinson}, it suffices in this case to show the existence of a natural isomorphism $\quot{\omega^{\fil\leq 0}}{f}$ or $\quot{\omega^{\fil\geq 0}}{f}$ and then deduce the existence of $\quot{\omega}{f}$ by extending. To this end note that since each of $\incl_{\fil\geq 0, \fil\leq 0}^{\fil\leq 0}$, $\incl_{\fil\leq 0}$, and $i$ are psueodnatural transformations we have invertible $2$-cells
\[
\begin{tikzcd}
T(Y) \ar[rr, ""{name = U}]{}{T(f)} \ar[d, swap]{}{\quot{\incl_{\fil\geq 0, \fil\leq 0}^{\fil\leq 0}}{Y}\circ\quot{i}{Y}} & & T(X) \ar[d]{}{\quot{\incl_{\fil\geq 0, \fil\leq 0}^{\fil\leq 0}}{X}\circ\quot{i}{X}} \\
F(Y)^{\fil\leq 0} \ar[rr, swap, ""{name = D}]{}{F^{\fil\leq 0}(f)} & & F(X)^{\fil\leq 0} \ar[from = U, to = D, Rightarrow, shorten <= 4pt, shorten >= 4pt]{}[description]{\quot{\left(\incl_{\fil\geq0,\fil\leq0}^{\fil\leq0} \circ i\right)}{f}}
\end{tikzcd}
\]
for every morphism $f:X \to Y$ in $\Cscr$. Because the vertical arrows are left adjoints to the functors $\quot{\omega}{Y} \circ \quot{\incl_{\fil\leq0}}{Y}$ and $\quot{\omega}{X}\circ\quot{\incl_{\fil\leq0}}{X}$, respectively, we then can also deduce the existence of invertible $2$-cells
\[
\begin{tikzcd}
	F(Y)^{\fil\leq 0} \ar[d, swap]{}{\quot{\omega}{Y} \circ \quot{\incl_{\fil\leq0}}{Y}} \ar[rr, ""{name = U}]{}{F^{\fil\leq 0}(f)} & & F(X)^{\fil\leq 0} \ar[d]{}{\quot{\omega}{X}\circ\quot{\incl_{\fil\leq0}}{X}} \\
	T(Y) \ar[rr, swap, ""{name = D}]{}{T(f)}  & & T(X) \ar[from = U, to = D, Rightarrow, shorten <= 4pt, shorten >= 4pt]{}[description]{\quot{\omega^{\fil\leq0}}{f}}
\end{tikzcd}
\]
Extending this to the existence of the natural isomorphisms $\quot{\omega}{f}$ then allow us to deduce the proposition below.
\begin{proposition}\label{Prop: Section Pseudocone Realization: Pseudocone Version of Beil omega}
Let $F:\Cscr^{\op} \to \fCat$ be a filtered triangulated pseudofunctor over a triangulated pseudofunctor $T:\Cscr^{\op} \to \fCat$. Then there is a pseudonatural transformation
\[
\begin{tikzcd}
	\Cscr^{\op} \ar[rr, bend left = 20, ""{name = U}]{}{F} \ar[rr, bend right = 20, swap, ""{name = D}]{}{T} & & \fCat \ar[from = U, to = D, Rightarrow, shorten <= 4pt, shorten >= 4pt]{}{\ul{\omega}}
\end{tikzcd}
\]
such that:
\begin{enumerate}
	\item There is an adjunction
	\[
\begin{tikzcd}
	\Cscr^{\op} \ar[rrrrrr, bend left = 40, ""{name = Up}]{}{F} \ar[rrrrrr, bend right = 40, swap, ""{name = Down}]{}{F} & & & & & & \fCat \ar[from = Down, to = Up, Rightarrow, shorten <= 4pt, shorten >= 4pt, bend right = 40, swap, ""{name = Left}]{}{\ul{\omega} \circ \ul{\incl}_{\fil\leq0}} \ar[from = Up, to = Down, Rightarrow, shorten <= 4pt, shorten >= 4pt, bend right = 40, swap, ""{name = Right}]{}{\ul{\incl}_{\fil\geq0,\fil\geq0}^{\fil\leq0} \circ \ul{i}} \ar[from = Right, to = Left, symbol = \dashv]
\end{tikzcd}
	\]
	in $\Bicat(\Cscr^{\op},\fCat)$.
\item There is an adjunction
	\[
	\begin{tikzcd}
		\Cscr^{\op} \ar[rrrrrr, bend left = 40, ""{name = Up}]{}{F} \ar[rrrrrr, bend right = 40, swap, ""{name = Down}]{}{F} & & & & & & \fCat \ar[from = Down, to = Up, Rightarrow, shorten <= 4pt, shorten >= 4pt, bend right = 40, swap, ""{name = Left}]{}{\ul{\incl}_{\fil\geq0,\fil\geq0}^{\fil\geq0} \circ \ul{i}} \ar[from = Up, to = Down, Rightarrow, shorten <= 4pt, shorten >= 4pt, bend right = 40, swap, ""{name = Right}]{}{\ul{\omega} \circ \ul{\incl}_{\fil\geq0}} \ar[from = Right, to = Left, symbol = \dashv]
	\end{tikzcd}
	\]
	in $\Bicat(\Cscr^{\op},\fCat)$.
\item The modification
\[
\begin{tikzcd}
	\Cscr^{\op} \ar[rrr, bend left = 40, ""{name = Up}]{}{F} \ar[rrr, bend right = 40, swap, ""{name = Down}]{}{T} & & & \fCat \ar[from = Up, to = Down, Rightarrow, shorten <= 4pt, shorten >= 4pt, bend left = 40, ""{name = Left}]{}{\ul{\omega} \circ \ul{s}} \ar[from = Up, to = Down, Rightarrow, shorten <= 4pt, shorten >= 4pt, bend right = 40, swap, ""{name = Right}]{}{\ul{\omega}} \ar[from = Right, to = Left, symbol = \underset{\ul{\omega} \ast \ul{\alpha}}{\Rrightarrow}]
\end{tikzcd}
\]
is an isomorphism in $\Bicat(\Cscr^{\op},\fCat)(F,F)$.
\item For every object $A$ of $\PC(F)^{\fil \leq 0}$ and $B$ of $\PC(F)^{\fil\geq 0}$, the map
\[
\PC(F)(A,B) \to \PC(T)(\omega(A),\omega(B))
\]
given object-locally by $\varphi \mapsto \quot{\omega}{X}(\varphi)$ is an isomorphism.
\end{enumerate}
\end{proposition}
\begin{proof}
The existence of $\ul{\omega}$ follows from the discussion prior ot the statement of the proposition. From here both (1) and (2) follow from the object-local adjunctions asserted by by Proposition \ref{Prop: Section Pseudocone Realization: Omega functors in Beilinson} and an appeal to \cite[{Theorem 4.1.17}]{Monograph}. Item (3) holds because each natural map $(\omega \ast \alpha)_X = \quot{\omega}{X}(\quot{\alpha}{X}_{-})$ is an isomorphism. Finally, (4) follows from the fact that each corresponding object-local map is an isomorphism.
\end{proof}

We now extend these constructions to the case when the filtered triangulated category $\PC(F)$ and the triangulated category $\PC(T)$ come equipped with $t$-structures which are compatible with the relevant pseudofunctorial constructions. This requires us to both recall what it means to have compatible $t$-structures on filtered triangulated categories as well as extend this to the pseudofunctorial and pseudoconical case of this paper.
	
\begin{definition}[{\cite[{Definition A.4}]{Beilinson}}]\label{Defn: Section Pseudocone Realization: Beilinson A4}
Let $\Ascr$ be a filtered triangulated category over a triangulated category $\Tscr$ and assume that both $\Ascr$ and $\Tscr$ come equipped with $t$-structures. We say that the $t$-structures and filtrations are compatible if the functor
\[
\incl \circ i:\Tscr \to \Ascr
\]
is $t$-exact and if there is an equivalence
\[
s(\Ascr^{t \leq 0}) \simeq \Ascr^{t \leq -1}
\]
which commutes with the corresponding inclusion functors into $\Ascr$ in the sense that there is an invertible $2$-cell:
\[
\begin{tikzcd}
s(\Ascr^{t \leq 0}) \ar[rr, ""{name = U}]{}{\incl} \ar[dr, swap]{}{\simeq} & & \Ascr \\
 & \Ascr^{t \leq -1} \ar[ur, swap]{}{\incl} \ar[from = U, to = 2-2, Rightarrow, shorten <= 4pt, shorten >= 4pt]{}{\cong}
\end{tikzcd}
\]
\end{definition}
\begin{remark}
In \cite[{Definition A.4}]{Beilinson} it is only required that $\Ascr^{t \leq -1} = s(\Ascr^{t \leq 0})$; however, we have asked a weaker condition because the strict equality is too, well, strict for the higher categorical formalism we use in this paper. It is also worth noting that the invertible $2$-cell above is not an additional ask; it exists by virtue of the given equivalence.
\end{remark}

Note that in the above definition we make the convention that we write $\Ascr^{t \leq n}$ and $\Ascr^{t \geq n}$ for the $t$-structure graded subcategories of $\Ascr$ and $\Ascr^{\fil \leq n}$ and $\Ascr^{\fil\geq n}$ for the filtration induced subcategories of the filtered triangulated category $\Ascr$. For the pseudofunctorial convention, we need to replace filtered triangulated categories equipped with a $t$-structure with truncated pseudofunctors (cf.\@ Definition \ref{Defn: Recollection: truncated pseudofunctor}) which are also filtered triangulated categories (cf.\@ Definition \ref{Defn: Section Pseudocone Realizations: fpseudofunctor}).

Let us discuss some of this structure which will be interacting. Assume that we have a filtered triangulated pseudofunctor $F:\Cscr^{\op} \to \fCat$ over a triangulated pseudofunctor $T:\Cscr^{\op} \to \fCat$ for which both $F$ and $T$ are also truncated pseuodofunctors. Now if we want to extend the first condition, namely asking for the functor $\Tscr \to \Ascr$ to be $t$-exact, we need not do any non-na{\"i}ve work. Because both $F$ and $T$ are assumed to be truncated, if we assume that the functors $\quot{\incl}{X} \circ \quot{i}{X}:T(X) \to F(X)$ are $t$-exact for every object $X \in \Cscr_0$, then for every morphism $f:X \to Y$, each functor in the diagram
\[
\begin{tikzcd}
T(Y) \ar[rr, ""{name = U}]{}{\quot{\incl}{Y} \circ \quot{i}{Y}} \ar[d, swap]{}{T(f)} & & F(Y) \ar[d]{}{F(f)} \\
T(X) \ar[rr, swap, ""{name = D}]{}{\quot{\incl}{X} \circ \quot{i}{X}} & & F(X)
\end{tikzcd}
\]
is $t$-exact. Consequently this gives us that the functor $\incl \circ i:\PC(T) \to \PC(F)$ is $t$-exact and so provides us with the correct notion, at least as far as this paper is concerned, of being inclusion compatible. The more difficult aspect lies in understanding and unwinding the equivalence condition. 

We begin our extension of the equivalence condition by first seeing if we can define a subpseudofunctor $S$ of $F$ determined by object-locally taking $S(X)$ to be the image of $F(X)^{t \leq 0}$ under the filtration shift functor $\quot{s}{X}$. On objects the definition of $S$ is essentially forced: we define, for every object $X$ in $\Cscr$,
\[
S(X) := \quot{s}{X}\left(F(X)^{t \leq 0}\right).
\]
To define the fibre functors $S(f):S(Y) \to S(X)$ we need to induce some map $\quot{s}{Y}(F(Y)^{t \leq 0}) \to \quot{s}{X}(F(X)^{t \leq 0})$. Here we first use the fact that since $F$ is a filtered triangulated functor, for every morphism $f:X \to Y$ we have that $\quot{s}{X} \circ F(f) = F(f) \circ \quot{s}{Y}$ so for every object $A$ of $F(Y)^{t \leq 0}$,
\[
F(f)\left(\quot{s}{Y}(A)\right) = \left(F(f) \quot{s}{Y}\right)(A) = \left(\quot{s}{X} \circ F(f)\right)(A) = \quot{s}{X}\left(F(f)\left(A\right)\right);
\]
as such if we know that $F(f):F(Y) \to F(X)$ restricts to a functor
\[ 
F(f)^{t \leq 0}:F(Y)^{t \leq 0} \to F(X)^{t \leq 0}
\]
we can define $S(f)$. However, because $F$ is also a truncated pseudofunctor, $F(f)$ \emph{does} in fact restrict so for every object $A$ of $F(Y)^{t \leq 0}$, $\quot{s}{X}(F(f)(A))$ is an object of $S(X)$. Thus we define the functor components by
\[
S(f) := F(f)|_{S(Y)}
\]
Composition then follows the rule, if $X \xrightarrow{f} Y \xrightarrow{g} Z$ is a pair of composable arrows,
\[
S(f) \circ S(g) = F(f)|_{S(Y)} \circ F(g)|_{S(Z)} = \left(F(f) \circ F(g)\right)|_{S(Z)}
\]
so we define the compositor isomorphisms by simply restricting the $\phi_{f,g}$ as appropriate. With $S$ constructed we note already that by construction that restricting $s$ to the subpseudofunctor $F^{t \leq 0}$ of $F$ gives a pseudonatural isomorphism
\[
\begin{tikzcd}
\Cscr^{\op} \ar[rrrr, bend left = 20, ""{name = U}]{}{F^{t \leq 0}} \ar[rrrr, bend right = 20, swap, ""{name = D}]{}{S} & & & & \fCat \ar[from = U, to = D, Rightarrow, shorten <= 4pt, shorten >= 4pt]{}{\ul{s}|_{F^{t \leq 0}}}
\end{tikzcd}
\]
with strict commutativity witnesses. Note that the commutativity witnesses are strict identities because of the definition of $S(f)$:
\[
\begin{tikzcd}
F(Y)^{t \leq 0} \ar[r]{}{\quot{s}{Y}} \ar[d, swap]{}{F(f)^{t \leq 0}} & S(Y) \ar[d]{}{S(f)} \ar[r, equals] & \quot{s}{Y}(F(Y)) \ar[d]{}{F(f)^{t \leq 0}|_{S(Y)}} \\
F(X) \ar[r, swap]{}{\quot{s}{X}} & S(X) \ar[r, equals] & \quot{s}{X}(F(X)^{t \leq 0})
\end{tikzcd}
\]

Now that we have built $S$, the most straightforward aspect of extending the equivalence condition in Definition \ref{Defn: Section Pseudocone Realization: Beilinson A4} is to simply require that we have a pseudonatural equivalence $\ul{e}:S \to F^{t \leq -1}$. What remains is determining the coherence(s) we need to ask $\ul{e}$ to satisfy. However, this is not too rough to determine. First, because both $S$ and $F^{t \leq -1}$ are subpseudofunctors of $F$ and $\ul{e}$ is a pseudonatural equivalence, is an invertible $2$-cell
\[
\begin{tikzcd}
S \ar[rr, ""{name = U}]{}{} \ar[dr, swap]{}{\ul{e}} & & F \\
 & F^{t \leq -1} \ar[ur] \ar[from = U, to = 2-2, Rightarrow, shorten <= 4pt, shorten >= 4pt]{}{\cong}
\end{tikzcd}
\]
by the existence of $\ul{e}$. Consequently, our primary concern lies in ensuring that the inclusions of $S$ vary compatibly with those of $F^{t \leq -1}$. This can be done by asking that, for every morphism $f:X \to Y$ in $\Cscr$ the pasting diagram
\[
\begin{tikzcd}
S(Y) \ar[r, ""{name = UL}]{}{\quot{e}{Y}} \ar[d, swap]{}{S(f)} & F(Y)^{t \leq -1} \ar[d]{}[description]{F(f)^{t \leq -1}} \ar[rr, ""{name = UR}]{}{\quot{\incl_{t \leq -1}}{Y}} & & F(Y) \ar[d]{}{F(f)} \\
S(X) \ar[r, swap, ""{name = DL}]{}{\quot{e}{X}} &  F(X)^{t \leq -1} \ar[rr, swap, ""{name = DR}]{}{\quot{\incl_{t \leq -1}}{X}} & & F(X) \ar[from = UR, to = DR, shorten <= 4pt, shorten >= 4pt, Rightarrow]{}[description]{\quot{\incl_{t \leq -1}}{f}} \ar[from = UL, to = DL, shorten <= 4pt, shorten >= 4pt, Rightarrow]{}[description]{\quot{e}{f}}
\end{tikzcd}
\]
is equivalent to the $2$-cell:
\[
\begin{tikzcd}
S(Y) \ar[rr, ""{name = U}]{}{\quot{\incl_{S}}{Y}} \ar[d, swap]{}{S(f)} & & F(Y) \ar[d]{}{F(f)} \\
S(X) \ar[rr, swap, ""{name = D}]{}{\quot{\incl_{S}}{X}} & & F(X) \ar[from = U, to = D, Rightarrow, shorten <= 4pt, shorten >= 4pt]{}{\quot{\incl_{S}}{f}}
\end{tikzcd}
\]
This is the condition we seek, as we see below.

\begin{definition}\label{Defn: Section Pseudocone Realization: Truncated compat}
Let $F:\Cscr^{\op} \to \fCat$ be a filtered triangulated pseudofunctor over a triangulated pseudofunctor $T:\Cscr^{\op} \to \fCat$ and assume that $F$ and $T$ are both truncated pseudofunctors as well. We say the truncation structure is compatible with the filtered structure if:
\begin{itemize}
	\item For all objects $X$ in $\Cscr$ the functor
	\[
	\quot{\incl}{X}\circ\quot{i}{X}:T(X) \to F(X)
	\]
	is $t$-exact;
	\item With $S$ the pseudofunctor defined prior to the definition, there is a pseudonatural equivalence $\ul{e}:S \to F^{t \leq -1}$ such that for any morphism $f:X \to Y$ in $\Cscr$, the pasting diagram
	\[
	\begin{tikzcd}
		S(Y) \ar[r, ""{name = UL}]{}{\quot{e}{Y}} \ar[d, swap]{}{S(f)} & F(Y)^{t \leq -1} \ar[d]{}[description]{F(f)^{t \leq -1}} \ar[rr, ""{name = UR}]{}{\quot{\incl_{t \leq -1}}{Y}} & & F(Y) \ar[d]{}{F(f)} \\
		S(X) \ar[r, swap, ""{name = DL}]{}{\quot{e}{X}} &  F(X)^{t \leq -1} \ar[rr, swap, ""{name = DR}]{}{\quot{\incl_{t \leq -1}}{X}} & & F(X) \ar[from = UR, to = DR, shorten <= 4pt, shorten >= 4pt, Rightarrow]{}[description]{\quot{\incl_{t \leq -1}}{f}} \ar[from = UL, to = DL, shorten <= 4pt, shorten >= 4pt, Rightarrow]{}[description]{\quot{e}{f}}
	\end{tikzcd}
	\]
	is equivalent to the $2$-cell:
	\[
	\begin{tikzcd}
		S(Y) \ar[rr, ""{name = U}]{}{\quot{\incl_{S}}{Y}} \ar[d, swap]{}{S(f)} & & F(Y) \ar[d]{}{F(f)} \\
		S(X) \ar[rr, swap, ""{name = D}]{}{\quot{\incl_{S}}{X}} & & F(X) \ar[from = U, to = D, Rightarrow, shorten <= 4pt, shorten >= 4pt]{}{\quot{\incl_{S}}{f}}
	\end{tikzcd}
	\]
\end{itemize}
\end{definition}
\begin{lemma}\label{Lemma: Section Pseudocone Realizations: Sanity Chec for Compat}
Let $F:\Cscr^{\op} \to \fCat$ be a filtered triangulated pseudofunctor over a triangulated pseudofunctor $T:\Cscr^{\op} \to \fCat$ and assume that $F$ and $T$ are truncated pseudofunctors as well. If the truncation structures on $F$ and $T$ are compatible with the filtered structure then the $t$-structures on $\PC(F)$ and $\PC(T)$ are compatible in the sense of Definition \ref{Defn: Section Pseudocone Realization: Beilinson A4}.
\end{lemma}
\begin{proof}
The first property we must prove is trivial to show, as that proving the functor $\incl \circ i:\PC(T) \to F(X)$ is $t$-exact is equivalent to checking that for every $X \in \Cscr_0$ the functor $\quot{\incl}{X} \circ \quot{i}{X}$ is $t$-exact.

In the second case note that because $\ul{e}:S \to F^{t\leq 1}$ is a pseudonatural equivalence, $e:\PC(S) \to \PC(F)^{t \leq -1}$ is an equivalence. That it commutes with the relevant inclusion functors is then immediate.
\end{proof}

Assume that we have a filtered triangulated pseudofunctor $F:\Cscr^{\op} \to \fCat$ over a triangulated pseudofunctor $T:\Cscr^{\op} \to \fCat$ and that $T$ is truncated. Since each category $T(X)$, for every object $X \in \Cscr_0$, is equipped with a $t$-structure by applying \cite[{Proposition A.5.i}]{Beilinson} we get that there is a unique $t$-structure on $F(X)$ for which the $t$-structures on $F(X)$ and $T(X)$ are compatible with the filtered structure. This $t$-structure is given by declaring that the objects of the $t \geq 0$ subcategory take the form
\[
\left(F(X)^{t \geq 0}\right)_0 = \left\lbrace A \in F(X)_0 \; : \; \quot{\ul{\gr}}{X}_{f}^{n}(A) \in \left(T(X)^{t \geq n}\right)_0 \right\rbrace
\]
and similarly the objects of the $t \leq 0$ subcategory take the form
\[
\left(F(X)^{t \geq 0}\right)_0 = \left\lbrace A \in F(X)_0 \; : \; \quot{\ul{\gr}}{Y}_{f}^{n}(A) \in \left(T(X)^{t \leq n}\right)_0 \right\rbrace.
\]
Our goal now is to show that these definitions make $F$ into a truncated pseudofunctor, as this will allow us to not only control the unique $t$-structure on $\PC(F)$ induced by \cite[{Proposition A.5.i}]{Beilinson}, but also help us translate the $t$-structure on $\PC(F)$ through the filtration and be important for the realization functor we desire to build. We will do this in a series of compartmentalized lemmas.
\begin{lemma}\label{Lemma: Pseudocone Realizations: The truncation functors and witness isos}
With $F$ and $T$ as above, for any morphism $f:X \to Y$ in $\Cscr$ there are functors $F(f)^{t\leq 0}:F(Y)^{t \leq 0} \to F(X)^{t \leq 0}$ and $F(f)^{t \geq 0}:F(Y)^{t \geq 0} \to F(X)^{t\geq0}$ and natural isomorphisms which fit into invertible $2$-cells
\[
\begin{tikzcd}
F(Y) \ar[rr, ""{name = U}]{}{F(f)} \ar[d, swap]{}{\quot{\tau_{\geq 0}}{Y}} & & F(X) \ar[d]{}{\quot{\tau_{\geq 0}}{X}} \\
F(Y)^{t \geq 0} \ar[rr, swap, ""{name = D}]{}{F(f)^{t \geq 0}} & & F(X)^{t \geq 0} \ar[from = U, to = D, Rightarrow, shorten <= 4pt, shorten >= 4pt]{}{\quot{\theta^{t \geq 0}}{f}}
\end{tikzcd}
\]
and:
\[
\begin{tikzcd}
	F(Y) \ar[rr, ""{name = U}]{}{F(f)} \ar[d, swap]{}{\quot{\tau_{\leq 0}}{Y}} & & F(X) \ar[d]{}{\quot{\tau_{\leq 0}}{X}} \\
	F(Y)^{t \leq 0} \ar[rr, swap, ""{name = D}]{}{F(f)^{t \leq 0}} & & F(X)^{t \leq 0} \ar[from = U, to = D, Rightarrow, shorten <= 4pt, shorten >= 4pt]{}{\quot{\theta^{t \leq 0}}{f}}
\end{tikzcd}
\]
\end{lemma}
\begin{proof}
The existence of the truncation functors $\quot{\tau_{\leq 0}}{Z}:F(Z) \to F(Z)^{t \leq 0}$ and $\quot{\tau_{\geq 0}}{Z}:F(Z) \to F(Z)^{t\geq 0}$ for every object $Z$ of $\Cscr$ follow from the fact that $F(Z)$ has a $t$-structure (cf.\@ \cite[{Proposition 1.3.3.i}]{BBD}). 

To induce the functors $F(f)^{t \leq 0}$ and $F(f)^{t \geq 0}$ we must show that if $A$ is an object of $F(Y)$ for which $A$ is an object of $F(Y)^{t \leq 0}$ then $F(f)(A)$ is an object of $F(X)^{t \leq 0}$, and similarly for the $t \geq 0$ case. This in turn amounts to showing that if $\quot{\ul{\gr}_f^n}{Y}(A) \in (T(Y)^{t \leq n})_0$ then $F(f)(A)$ has the property that $\quot{\ul{\gr}_f^n}{X}(F(f)A)$ is an object of $T(X)^{t \leq n}$. For this first note that because $\ul{\gr}_{f}^{n}:F \Rightarrow T$ is a pseudonatural transformation we have isomorphisms
\[
\left(\quot{\ul{\gr}_{f}^{n}}{X} \circ F(f)\right)(A) \cong \left(T(f) \circ \quot{\ul{\gr}_{f}^{n}}{Y}\right)(A) = T(f)\left(\left(\quot{\ul{\gr}_{f}^{n}}{Y}\right)(A)\right).
\]
Because $T$ is a truncated pseudofunctor, $T(f)$ is left and right $t$-exact and so necessarily preserves truncation degrees. Thus $\quot{\ul{\gr}_{f}^{n}}{Y}(A)$ is an object of $T(X)^{t \leq n}$. However, because $T(X)^{t \leq n}$ is a strictly full subcategory of $T(X)$, this implies that $\quot{\ul{\gr}_{f}^{n}}{X}\left(F(f)(A)\right)$ is an object of $T(X)^{t \leq n}$ as well and hence that we have our desired functor $F(f)^{t \leq 0}$. Finally, the existence of the isomorphism $\quot{\theta^{t \leq 0}}{f}$ follows from the fact that each diagram
\[
\begin{tikzcd}
F(Y)^{t \leq 0} \ar[r]{}{F(f)^{t \leq 0}} \ar[d, swap]{}{\quot{\incl_{t \leq 0}}{Y}} & F(X)^{t \leq 0} \ar[d]{}{\quot{\incl_{t \leq 0}}{X}} \\
F(Y) \ar[r, swap]{}{F(f)} & F(X)
\end{tikzcd}
\]
commutes strictly and because we have adjunctions $\quot{\incl_{t \leq 0}}{Y} \dashv \quot{\tau_{\leq 0}}{Y}$ and $\quot{\incl_{t \leq 0}}{X} \dashv \quot{\tau_{\leq 0}}{X}$. The existence of the functors $F(f)^{t \geq 0}$ and the commutativity witnesses $\quot{\theta^{t \geq 0}}{f}$ follows similarly save that we use the adjunctions $\quot{\tau_{\geq 0}}{Y} \dashv \quot{\incl_{t \geq 0}}{Y}$ and $\quot{\tau_{\geq 0}}{X} \dashv \quot{\incl_{t \geq 0}}{X}$ instead.
\end{proof}
\begin{lemma}\label{Lemma: Section Pseudocone Realizations: Witness isos are pseudonat}
The natural isomorphisms $\quot{\theta^{t \geq 0}}{f}$ and $\quot{\theta^{t \leq 0}}{f}$ vary pseudonaturally in $\Cscr^{\op}$.
\end{lemma}
\begin{proof}
We only prove this for the $t \geq 0$ case, as the $t \leq 0$ case holds similarly. Begin by letting $X \xrightarrow{f} Y \xrightarrow{g} Z$ be a composable pair of morphisms in $\Cscr$ and consider the pasting diagram:
\begin{equation}\label{Eqn: Pasto}
\begin{tikzcd}
F(Z) \ar[rr, ""{name = UL}]{}{F(g)} \ar[dd, swap]{}{\quot{\tau^{\geq 0}}{Z}} & & F(Y) \ar[dd]{}[description]{\quot{\tau^{\geq 0}}{Y}} \ar[rr, ""{name = UR}]{}{F(f)} & & F(X) \ar[dd]{}{\quot{\tau^{\geq 0}}{X}} \\ 
\\
F(Z)^{t \geq 0} \ar[rr, swap, ""{name = DL}]{}{F(g)^{t \geq 0}} \ar[rrrr, bend right = 40, swap, ""{name = DD}]{}{F(g \circ f)^{t \geq 0}} & & F(Y) \ar[rr, swap, ""{name = DR}]{}{F(f)^{t \geq 0}} & & F(X)^{t \geq 0} \ar[from = UL, to = DL, Rightarrow, shorten <= 4pt, shorten >= 4pt]{}{\quot{\theta^{t \geq 0}}{g}} \ar[from = UR, to = DR, Rightarrow, shorten <= 4pt, shorten >= 4pt]{}{\quot{\theta^{t \geq 0}}{f}} \ar[from = 3-3, to = DD, Rightarrow, shorten <= 4pt, shorten >= 4pt]{}{\quot{\phi_{f,g}}{t \geq 0}}
\end{tikzcd}
\end{equation}
Because the isomorphisms $\quot{\theta^{t \geq 0}}{f}$ and $\quot{\theta^{t \geq 0}}{g}$ are constructed from the adjunctions $\quot{\tau_{\geq 0}}{X} \dashv \quot{\incl_{t \geq 0}}{X}$ and $\quot{\tau_{\geq 0}}{Y} \dashv \quot{\incl_{t \geq 0}}{Y}$ together with the pseudonaturality of $\incl_{t \geq 0}$, it suffices to show that the pseudonaturality pasting requoriements come from the same diagrams on the adjoint side, as then the corresponding isomorphisms are members of a contractible groupoid of isomorphisms. For this, however, we simply note that because the inclusion functors vary pseudonaturally (and in fact strictly in this case) we have that the pasting diagrazm
\[
\begin{tikzcd}
F(Z)^{t \geq 0} \ar[rr, ""{name = UL}]{}{F(g)^{t \geq 0}} \ar[dd, swap]{}{\quot{\incl_{t \geq 0}}{Z}} & & F(Y)^{t \geq 0} \ar[rr, ""{name = UR}]{}{F(f)^{t \geq 0}} \ar[dd]{}[description]{\quot{\incl_{t \geq 0}}{Y}} & & F(X)^{t \geq 0} \ar[dd]{}{\quot{\incl_{t \geq 0}}{X}} \\
 \\
F(Z) \ar[rr, swap]{}{F(g)} \ar[rrrr, swap, bend right = 40, ""{name = DD}]{}{F(g \circ f)} & & F(Y) \ar[rr, swap]{}{F(f)} & & F(X) \ar[from = 3-3, to = DD, Rightarrow, shorten <= 4pt, shorten >= 4pt]{}{\quot{\phi_{f,g}}{F}}
\end{tikzcd}
\]
is equal to the pasting diagram:
\[
\begin{tikzcd} 
 & F(Y)^{t \geq 0} \ar[dr]{}{F(f)^{t \geq 0}}	\\
F(Z)^{t \geq 0} \ar[d, swap]{}{\quot{\incl_{t \geq 0}}{Z}} \ar[rr, ""{name = U}]{}[description]{F(g \circ f)^{t \geq 0}} \ar[ur]{}{F(g)^{t \geq 0}} & {} & F(X)^{t \geq 0} \ar[d]{}{\quot{\incl_{t \geq 0}}{X}} \\
F(Z) \ar[rr, swap]{}{F(g \circ f)} & & F(X) \ar[from = 1-2, to = U, Rightarrow, shorten >= 4pt, shorten <= 4pt]{}{\quot{\phi_{f,g}}{t \geq 0}}
\end{tikzcd}
\]
Thus, after potentially redefining the $\theta$ along a choice of contraction of the groupoid of isomorphisms, we find that the pasting diagram expressed in Diagram \ref{Eqn: Pasto} conicides with the pasting diagram:
\[
\begin{tikzcd}
	& F(Y) \ar[dr]{}{F(f)} \\
F(Z) \ar[d, swap]{}{\quot{\tau_{\geq 0}}{Z}} \ar[ur]{}{F(g)} \ar[rr, ""{name = U}]{}{F(g \circ f)} & & F(X) \ar[d]{}{\quot{\tau_{\geq 0}}{X}} \\
F(Z)^{t \geq 0} \ar[rr, swap, ""{name = D}]{}{F(g \circ f)^{t \geq 0}} & & F(X)^{t \geq 0} \ar[from = 1-2, to = U, Rightarrow, shorten <= 4pt, shorten >= 4pt]{}{\quot{\phi_{f,g}}{F}} \ar[from = U, to = D, Rightarrow, shorten <= 4pt, shorten >= 4pt]{}{\quot{\theta^{t \geq 0}}{g \circ f}}
\end{tikzcd}
\]
\end{proof}
\begin{corollary}\label{Cor: Section Pseudocone Realizations: Truncations are neato}
Assume that we have a filtered triangulated pseudofunctor $F:\Cscr^{\op} \to \fCat$ over a triangulated pseudofunctor $T:\Cscr^{\op} \to \fCat$ and that $T$ is truncated. The assignments $F^{t \leq 0}$ and $F^{t \geq 0}$ induce pseudofunctors with adjoints
\[
\begin{tikzcd}
F^{t \geq 0} \ar[rrr, swap, bend right = 20, ""{name = D}]{}{\ul{\incl}_{t \geq 0}} & & & F \ar[lll, bend right = 20, swap, ""{name = U}]{}{\ul{\tau}_{\geq 0}} \ar[from = U, to = D, symbol = \dashv]
\end{tikzcd}
\]
and
\[
\begin{tikzcd}
F \ar[rrr, bend right = 20, swap, ""{name = D}]{}{\ul{\tau}_{\leq 0}} & & & F^{t \leq 0} \ar[lll, bend right = 20, swap, ""{name = U}]{}{\ul{\incl}_{t \leq 0}} \ar[from = U, to = D, symbol = \dashv]
\end{tikzcd}
\]
in the $2$-category $\Bicat(\Cscr^{\op},\fCat)$. In particular, $F$ is a truncated pseudofunctor with each $t$-structure on $F(X)$ the unique $t$-structure compatible with the $t$-structure on $T(X)$.
\end{corollary}
\begin{proof}
Lemma \ref{Lemma: Pseudocone Realizations: The truncation functors and witness isos} shows that $F^{t \leq 0}$ and $F^{t \geq 0}$ induce pseudofunctors and Lemma \ref{Lemma: Section Pseudocone Realizations: Witness isos are pseudonat} shows that the collections $(\quot{\tau_{\geq 0}}{X},\quot{\theta^{t \geq 0}}{f})_{X \in \Cscr_0, f \in \Cscr_1}$ and $(\quot{\tau_{\leq 0}}{X}, \quot{\theta^{t \leq 0}}{f})_{X \in \Cscr_0, f \in \Cscr_1}$ are pseudonatural transformations. From here the adjunctions follow by virtue of \cite[{Theorem 4.1.17}]{Monograph} and the facts that $\quot{\tau_{\geq 0}}{X} \dashv \quot{\incl_{t \geq 0}}{X}$ and $\quot{\incl_{t \leq 0}}{X} \dashv \quot{\tau_{\leq 0}}{X}$ for all objects $X$ in $\Cscr_0$. Finally that $F$ is truncated is immediate from the pseudonaturality $\ul{\tau_{\leq 0}}$ and $\ul{\tau_{\geq 0}}$ while the statement regarding compatibility follows by construction.
\end{proof}

This allows us to deduce the following proposition regarding the $t$-structure on the filtered triangulated category $\PC(F)$ over $\PC(T)$ when $T$ is truncated.
\begin{proposition}\label{Prop: Section Pseudocone Realizations: The nice description of stuffs for truncated and filtered}
Assume that we have a filtered triangulated pseudofunctor $F:\Cscr^{\op} \to \fCat$ over a triangulated pseudofunctor $T:\Cscr^{\op} \to \fCat$ and that $T$ is truncated. Then the unique $t$-structure on $\PC(T)$ compatible with the $t$-structure on $\PC(T)$ satisfies
\begin{align*}
	\left(\PC(F)^{t \geq 0}\right)_0 &= \left\lbrace A \; : \; \gr_{\fil}^{n}(A) \in \left(\PC(T)^{t \geq n}\right)_0\right\rbrace = \left\lbrace A \; : \; \quot{\ul{\gr}_f^n}{X}\left(\quot{A}{X}\right) \in \left(T(X)^{t \geq n}\right)_0 \right\rbrace \\
	&= \left\lbrace A \; : \; \quot{\ul{\gr}_f^n}{X}\left(\quot{A}{X}\right) \in \left(F(X)^{t \geq 0}\right)_0 \right\rbrace
\end{align*}
and
\begin{align*}
	\left(\PC(F)^{t \leq 0}\right)_0 &= \left\lbrace A \; : \; \gr_{\fil}^{n}(A) \in \left(\PC(T)^{t \leq n}\right)_0\right\rbrace = \left\lbrace A \; : \; \quot{\ul{\gr}_f^n}{X}\left(\quot{A}{X}\right) \in \left(T(X)^{t \leq n}\right)_0 \right\rbrace \\
&= \left\lbrace A \; : \; \quot{\ul{\gr}_f^n}{X}\left(\quot{A}{X}\right) \in \left(F(X)^{t \leq 0}\right)_0 \right\rbrace
\end{align*}
\end{proposition}
\begin{proof}
In both $t \geq 0$ and $t \leq 0$ cases, the first two equalities are immediate from the fact that $T$ is truncated  and \cite[{Proposition A.5.i}]{Beilinson}. The last equality follows from the fact that $F$ is truncated by Corollary \ref{Cor: Section Pseudocone Realizations: Truncations are neato} with truncation induced object-locally by the unique $t$-structure on $F(X)$ compatible with the $t$-structure on $T(X)$.
\end{proof}

We now want to build a cohomological functor $H_{\fil}:\PC(F) \to \PC(\Ch^b(T^{\heartsuit}))$ by pseudoconifying the construction of the functor $H_{\fil}$ of \cite[{Proposition A.5.b}]{Beilinson}. Because those functors allow us to construct cohomological functors
\[
\quot{H_{\fil}}{X}:F(X) \to \Ch^b\left(T(X)^{\heartsuit}\right)
\]
for all objects $X$ of $\Cscr$, our primary obstructions in defining the functor $H_{\fil}:\PC(F) \to \PC\left(\Ch^b(T)^{\heartsuit}\right)$ comes down to showing that the differentials of the $\quot{H_{\fil}}{X}$ are compatible with pseudonatural information, as well as that they are suitably natural. To this end we first recall the construction of the cohomological functors
\[
\quot{H_{\fil}}{X}:F(X) \to \Ch^b\left(T(X)^{\heartsuit}\right).
\]
\begin{proposition}[{\cite[{Proposition A.5.b}]{Beilinson}}]\label{Prop: Section Pseudocone Realizations: Beilinsion A5b}
Assume that $\Cscr$ is a filtered triangulated category over a triangulated category $\Ascr$ and assume that they have compatible $t$-structures with hearts $\Cscr^{\heartsuit}$ and $\Ascr^{\heartsuit}$, respectively. Then there is a cohomological functor $H_{\fil}:\Cscr \to \Ch^b(\Ascr^{\heartsuit})$ such that $H_{\fil}|_{\Cscr^{\heartsuit}}$ is an equivalence of categories.
\end{proposition}
\begin{proof}[Sketch]
We provide a sketch of the proposition because we will need to use it later. If $X$ is an object of $\Cscr$ then the objects of the cochain complex $(H_{\fil}(X),\partial)$ are given by defining
\[
H_{\fil}(X)^{n} := H^n\left(\gr_{\fil}^{n}(X)\right)
\]
where $H:\Ascr \to \Ascr^{\heartsuit}$ is the cohomology functor induced by the $t$-structure on $\Ascr$. The differentials $\partial_n:H_{\fil}(X)^{n} \to H_{\fil}(X)^{n+1}$ are induced as follows. For each $n \in \Z$, consider the object $\left(\Sigma_{\fil \geq n+1} \circ \Sigma_{\fil \leq n} \right)(X);$ cf.\@ Remark \ref{Remark: The Cap sigma notation} for the definition of the $\Sigma$ functors. By proceeding as in the usual $t$-structure case by using the adjoint calculus as in \cite{BBD}, we can show that for all natural numbers $m, n \in \N$ there are natural isomorphisms 
\begin{align}
\Sigma_{\fil\leq n} &\cong	\Sigma_{\fil\leq n} \circ \Sigma_{\fil\leq m}\label{Eqn: Leq niso} \\
\Sigma_{\fil\geq m} &\cong \Sigma_{\fil\geq m}\ \circ \Sigma_{\fil\geq n}\label{Eqn: Geq niso} \\
\Sigma_{\fil\geq n} \circ \Sigma_{\fil\leq m} &\cong \Sigma_{\fil\leq m} \circ \Sigma_{\fil \geq n}.\label{Eqn: Mixed niso}
\end{align} 
whenever $n \leq m$. Now, apply \cite[{Proposition A.3.ii}]{Beilinson} to give a distinguished triangle
\[
\begin{tikzcd}
\left(\Sigma_{\fil \leq n+1} \circ \Sigma_{\fil \geq n+1}\right)(X) \ar[r]{}{} & \left(\Sigma_{\fil \leq n+1} \circ \Sigma_{\fil\geq n}\right)(X) \ar[d] \\ 
\left(\Sigma_{\fil \leq n+1} \circ \Sigma_{\fil \geq n+1}\right)(X)[1] & \left(\Sigma_{\fil \leq n} \circ \Sigma_{\fil\geq n}\right)(X) \ar[l]{}{d}
\end{tikzcd}
\]
where we have implicitly used the various natural isomorphisms in Equations \ref{Eqn: Leq niso}, \ref{Eqn: Geq niso}, \ref{Eqn: Mixed niso} in various places. Now applying the functor $\omega$ gives a distinguished triangle
\[
\begin{tikzcd}
	\left(\omega \circ\Sigma_{\fil \leq n+1} \circ \Sigma_{\fil \geq n+1}\right)(X) \ar[r]{}{} & \left(\omega \circ\Sigma_{\fil \leq n+1} \circ \Sigma_{\fil\geq n}\right)(X) \ar[d] \\ 
	\left(\omega \circ\Sigma_{\fil \leq n+1} \circ \Sigma_{\fil \geq n+1}\right)(X)[1] & \left(\omega \circ\Sigma_{\fil \leq n} \circ \Sigma_{\fil\geq n}\right)(X) \ar[l]{}{\omega(d)}
\end{tikzcd}
\]
which, after manipulating the adjunctions $\omega$ satisfies with the various  inclusion functors, gives rise to a distinguished triangle
\[
\begin{tikzcd}
\left(j \circ s^{-n} \circ\Sigma_{\fil \leq n+1} \circ \Sigma_{\fil \geq n+1}\right)(X) \ar[r]{}{} & \left(j \circ s^{-n} \circ\Sigma_{\fil \leq n+1} \circ \Sigma_{\fil\geq n}\right)(X) \ar[d] \\ 
\left(j \circ s^{-n} \circ\Sigma_{\fil \leq n+1} \circ \Sigma_{\fil \geq n+1}\right)(X)[1] & \left(j \circ s^{-n} \circ\Sigma_{\fil \leq n} \circ \Sigma_{\fil\geq n}\right)(X) \ar[l]{}{\widetilde{d}}
 \end{tikzcd}
\]
are in $\Ascr$. Consequently, applying the cohomology functor $H^{n}$ gives rise to the differential
\[
\partial_n = H^{n}(\widetilde{d}):H_{\fil}(X)^n \to H_{\fil}(X)^{n+1}
\]
as the connecting morphism of the corresponding long exact cohomology sequence. That these maps are differentials is then routine as they come from the long exact sequence. Finally the statement regarding $H_{\fil}|_{\Cscr^{\heartsuit}}$ being an equivalence is routine and omitted.
\end{proof}

Note that by Lemma \ref{Lemma: Section Pseudocone Realization: Graded PC functor} and Propositions \ref{Prop: Section Pseudocone Realizations: Dist triangle pseudocone version} and \ref{Prop: Section Pseudocone Realization: Pseudocone Version of Beil omega}, we know that the constructions used above can extended to the pseudoconical situation. In particular, we can use the functors $H_{\fil}$ to construct a pseudonatural transformation that lifts to a cohomological functor $\PC(F) \to \PC(\Ch^b(T^{\heartsuit}))$. Note that technically speaking we should check that $\Ch^b(T^{\heartsuit})$ determines a pseudofunctor, but this is routine by virtue of the fact that $T^{\heartsuit}$ is a pseudofunctor and the $\Ch^b(-)$ construction is (strictly) functorial with its assignment on morphisms given by prolongation.
\begin{proposition}\label{Prop: Section Pseudocone Realization: Cohomological functor but now its pseudonat}
Let $F:\Cscr^{\op} \to \fCat$ be a filtered triangulated pseudofunctor over triangulated pseudofunctor $T:\Cscr^{\op} \to \fCat$ and assume that $T$ and $F$ are truncated with compatible $t$-structures. Then there is a pseudonatural transformation
\[
\begin{tikzcd}
\Cscr^{\op} \ar[rr, bend left = 20, ""{name = U}]{}{F} \ar[rr, bend right = 20, swap, ""{name = D}]{}{\Ch^b(T^{\heartsuit})} & & \fCat \ar[from = U, to = D, Rightarrow, shorten <= 4pt, shorten >= 4pt]{}{\ul{H}_{\fil}}
\end{tikzcd}
\]
such that each object functor $\quot{\ul{H}_{\fil}}{X}$ is a cohomological functor.
\end{proposition}
\begin{proof}
For each object $X$ of $\Cscr$ define the functor $\quot{\ul{H}_{\fil}}{X}$ to be the cohomology functor
\[
\quot{H_{\fil}}{X}:F(X) \to \Ch^b\left(T(X)^{\heartsuit}\right)
\]
of Proposition \ref{Prop: Section Pseudocone Realizations: Beilinsion A5b}. To define the witness isomorphisms $\quot{\ul{H}_{\fil}}{f}$ for a morphism $f:X \to Y$ in $\Cscr$, we note that for every object $A$ of $T(Y)$, our desired natural isomorphism $(\quot{\ul{H}_{\fil}}{f})_{A}$ is an isomorphism in the category $\Ch^b(T(X)^{\heartsuit})$. As such, we must define isomorphisms $(\quot{\ul{H}_{\fil}}{f})_A^n$ for all $n \in \Z$ and then show that these commute with the differentials. To this end fix an $n \in \Z$ and define $(\quot{\ul{H}_{\fil}}{f})_{-}^n$ as follows. 

Consider that since $F$ is filtered over $T$, there is an invertible $2$-cell
\[
\begin{tikzcd}
F(Y) \ar[r, ""{name = U}]{}{\quot{\gr_{\fil}^n}{Y}} \ar[d, swap]{}{F(f)} & T(Y) \ar[d]{}{T(f)} \\
F(X) \ar[r, swap, ""{name = D}]{}{\quot{\gr_{\fil}^{n}}{X}} & T(X) \ar[from = U, to = D, Rightarrow, shorten <= 4pt, shorten >= 4pt]{}{\quot{\ul{\gr}_{\fil}^{n}}{f}}
\end{tikzcd}
\]
by Lemma \ref{Lemma: Section Pseudocone Realization: Graded PC functor}. Furthermore, since $T$ is truncated, there is an invertible $2$-cell
\[
\begin{tikzcd}
T(Y) \ar[r, ""{name = U}]{}{\quot{H_T^n}{Y}} \ar[d, swap]{}{T(f)} & T(Y)^{\heartsuit} \ar[d]{}{T(f)^{\heartsuit}} \\
T(X) \ar[r, swap, ""{name = D}]{}{\quot{H_T^n}{X}} & T(X)^{\heartsuit} \ar[from = U, to = D, Rightarrow, shorten <= 4pt, shorten >= 4pt]{}{\quot{H_T^n}{f}}
\end{tikzcd}
\]
because $T$ is truncated. Pasting the $2$-cells together gives the invertible $2$-cell:
\[
\begin{tikzcd}
F(Y) \ar[rr, ""{name = U}]{}{\quot{H_{T}^n}{Y} \circ \quot{\gr_{\fil}^n}{Y}} \ar[d, swap]{}{F(f)} & & T(Y)^{\heartsuit} \ar[d]{}{T(f)^{\heartsuit}} \\
F(X) \ar[rr, swap, ""{name = D}]{}{\quot{H_T^n}{X} \circ \quot{\gr_{\fil}^{n}}{X}} & & T(X)^{\heartsuit} \ar[from = U, to = D, Rightarrow, shorten <= 4pt, shorten >= 4pt]{}[description]{\quot{H_T^n}{f} \ast \quot{\gr_{\fil}^n}{f}}
\end{tikzcd}
\]
Now observe that we can rewrite the above $2$-cell as
\[
\begin{tikzcd}
	F(Y) \ar[rr, ""{name = U}]{}{(\quot{H_{\fil}}{Y})^n} \ar[d, swap]{}{F(f)} & & T(Y)^{\heartsuit} \ar[d]{}{T(f)^{\heartsuit}} \\
	F(X) \ar[rr, swap, ""{name = D}]{}{(\quot{H_{\fil}}{X})^n} & & T(X)^{\heartsuit} \ar[from = U, to = D, Rightarrow, shorten <= 4pt, shorten >= 4pt]{}[description]{\quot{H_T^n}{f} \ast \quot{\gr_{\fil}^n}{f}}
\end{tikzcd}
\]
which we then use to define the degree $n$-component of our candidate isomorphism: 
\[
(\quot{\ul{H}_{\fil}}{f})_{-}^n := \quot{H_T^n}{f} \ast \quot{\gr_{\fil}^n}{f}.
\]
Fix an object $A$ of $F(Y)$. Observe that by the construction of the distinguished triangles in the proof of Proposition \ref{Prop: Section Pseudocone Realizations: Dist triangle pseudocone version} used in the proof of Proposition \ref{Prop: Section Pseudocone Realizations: Beilinsion A5b}, the triangles
\[
\begin{tikzcd}
\left(\quot{\Sigma_{\fil \leq n+1}}{X} \circ \quot{\Sigma_{\fil \geq n+1}}{X}\right)\left(F(f)(A)\right) \ar[r]{}{} & \left(\quot{\Sigma_{\fil \leq n+1}}{X} \circ \quot{\Sigma_{\fil\geq n}}{X}\right)\left(F(f)(A)\right) \ar[d] \\ 
\left(\quot{\Sigma_{\fil \leq n+1}}{X} \circ \quot{\Sigma_{\fil \geq n+1}}{X}\right)\left(F(f)(A)\right)[1] & \left(\quot{\Sigma}{X}_{\fil \leq n} \circ \quot{\Sigma}{X}_{\fil\geq n}\right)\left(F(f)(A)\right) \ar[l]{}{\quot{d_{F(f)A}}{X}}
\end{tikzcd}
\]
and
\[
\begin{tikzcd}
	F(f)\left(\quot{\Sigma}{Y}_{\fil \leq n+1} \circ \quot{\Sigma}{Y}_{\fil \geq n+1}\right)(A) \ar[r]{}{} & \left(\quot{\Sigma}{Y}_{\fil \leq n+1} \circ \quot{\Sigma}{Y}_{\fil\geq n}\right)(A) \ar[d] \\ 
	\left(\quot{\Sigma}{Y}_{\fil \leq n+1} \circ \quot{\Sigma}{Y}_{\fil \geq n+1}\right)(A)[1] & \left(\quot{\Sigma}{Y}_{\fil \leq n} \circ \quot{\Sigma}{Y}_{\fil\geq n}\right)(A) \ar[l]{}{F(f)\quot{d_A}{Y}}
\end{tikzcd}
\]
are object-wise isomorphic with non-labeled arrows induced by the $\incl$ and $\sigma$ adjunctions, it follows that
%
%
the corresponding distinguished triangles are homotopic in $T(X)$ after applying the $\gr_{\fil}^n$ functors via $\quot{\gr_{\fil}^n}{f}$. However, this implies that because the functor $H_T^n$ is is a cohomology functor that we induce an isomorphism of long exact sequences and hence a commuting square:
\[
\begin{tikzcd}
\left(T(f) \circ \quot{H_{\fil}}{Y}\right)(A)^n \ar[rr]{}{\partial} \ar[d, swap]{}{(\quot{\ul{H}_{\fil}}{f})_A^n} & & \left(T(f) \circ \quot{H_{\fil}}{Y}\right)(A)^n \ar[d]{}{(\quot{\ul{H}_{\fil}}{f})^{n+1}_{A}} \\
\left(\quot{H_{\fil}}{X} \circ F(f)\right)(A)^n \ar[rr, swap]{}{\partial} & & \left(\quot{H_{\fil}}{X} \circ F(f)\right)(A)^{n+1}
\end{tikzcd}
\]
It thus follows that $\ul{H_{\fil}}$ is a pseudonatural transformation. That each corresponding functor $\quot{\ul{H}_{\fil}}{X}$ is a cohomological functor is the content of Proposition \ref{Prop: Section Pseudocone Realizations: Beilinsion A5b}.
\end{proof}
\begin{corollary}\label{Cor: Section Pseudocone Realization: The pseudocone equiv we want}
Let $F:\Cscr^{\op} \to \fCat$ be a filtered triangulated pseudofunctor over triangulated pseudofunctor $T:\Cscr^{\op} \to \fCat$ and assume that $T$ and $F$ are truncated with compatible $t$-structures. There is a pseudonatural equivalence
\[
\ul{H}_{\fil}|_{F^{\heartsuit}}:F^{\heartsuit} \Rightarrow \Ch^b\left(T^{\heartsuit}\right).
\]
\end{corollary}
\begin{proof}
Proposition \ref{Prop: Section Pseudocone Realizations: Beilinsion A5b} shows that for every object $X$ of $\Cscr$, the functor
\[
\quot{\ul{H}_{\fil}}{X}|_{F(X)^{\heartsuit}}:F(X)^{\heartsuit} \to \Ch^b\left(T(X)^{\heartsuit}\right)
\]
is an equivalence of categories. However, since $F$ is truncated and hence each fibre functor $F(f)$ is $t$-exact, it follows that the pseudonatural transformation $\ul{H}_{\fil}$ restricts to a pseudonatural transformation 
\[
\ul{H}_{\fil}|_{F^{\heartsuit}}:F^{\heartsuit} \Rightarrow \Ch^b\left(T^{\heartsuit}\right).
\]
From the given equivalences of categories, the result follows.
\end{proof}
\begin{corollary}\label{Cor: Section Pseudocone Realization: The cohomology boi}
Let $F:\Cscr^{\op} \to \fCat$ be a filtered triangulated pseudofunctor over triangulated pseudofunctor $T:\Cscr^{\op} \to \fCat$ and assume that $T$ and $F$ are truncated with compatible $t$-structures. There is a cohomological functor
\[
H_{\fil}:\PC(F) \to \PC\left(\Ch^b\left(T^{\heartsuit}\right)\right)
\]
which restricts to an equivalence of categories $\PC(F)^{\heartsuit} \simeq \PC(\Ch^b(T^{\heartsuit}))$.
\end{corollary}
\begin{proof}
The existence of the functor $H_{\fil}$ follows by setting $H_{\fil} := \PC(\ul{H}_{\fil})$, where $\ul{H}_{\fil}$ is the pseudonatural transformation of Proposition \ref{Prop: Section Pseudocone Realization: Cohomological functor but now its pseudonat}. That this is a cohomological functor follows from the fact that each $\quot{H_{\fil}}{X}$ is a cohomological functor. Finally, that this restricts to an equivalence is routine from the fact that $\ul{H}_{\fil}|_{F^{\heartsuit}}$ is an equivalence by Corollary \ref{Cor: Section Pseudocone Realization: The pseudocone equiv we want} and from the Change of Heart Theorem (cf.\@ \cite[{Theorem 5.1.24}]{Monograph}) giving that $\PC(F^{\heartsuit}) = \PC(F)^{\heartsuit}$.
\end{proof}

We now can proceed as in the remainder of \cite{Beilinson} to construct the pseudocone realization functor. Let $F:\Cscr^{\op} \to \fCat$ be a filtered triangulated pseudofunctor over triangulated pseudofunctor $T:\Cscr^{\op} \to \fCat$ and assume that $T$ and $F$ are truncated with compatible $t$-structures. Let $I:\Ch^b(T^{\heartsuit}) \to F^{\heartsuit}$ be an inverse pseudonatural equivalence to $H_{\fil}|_{F^{\heartsuit}}$. Define the pseudonatural transformation
\[
	\begin{tikzcd}
	\Ch^b\left(T^{\heartsuit}\right) \ar[drr, swap]{}{\widetilde{\rea}} \ar[r]{}{I} & F^{\heartsuit} \ar[r]{}{\incl} & F \ar[d]{}{\omega} \\
	& & T
	\end{tikzcd}
\]
and consider the composite
\[
\begin{tikzcd}
\Ch^b(T^{\heartsuit}) \ar[r]{}{\widetilde{\rea}} & T \ar[r]{}{H_T} & T^{\heartsuit}
\end{tikzcd}
\]
in $\Bicat(\Cscr^{\op},\fCat)$. Now define the subpseudofunctor $Q$ of $\Ch^b(T^{\heartsuit})$ defined as follows:
\begin{itemize}
	\item For each object $X$ of $\Cscr$, the category $Q(X) := {}^{t}\mathsf{QIso}(\Ch^b(T(X)^{\heartsuit}))$ is the category of quasi-isomorphisms of $\Ch^b(T(X)^{\heartsuit})$, i.e., the category whose objects are the objects of $\Ch^b(T(X)^{\heartsuit})$ but whose morphisms are precisely the quasi-isomorphisms of $\Ch^b(T(X)^{\heartsuit})$.
	\item For each morphism $f:X \to Y$ of $\Cscr$, the fibre functor $Q(f)$ is defined to be the prolongation $\Ch^b(T(f))$ as restricted to $Q(X)$. Because $T$ is truncated and so each functor $T(f)$ is $t$-exact, it follows that $\Ch^b(T(f))$ preserves quasi-isomorphisms.
	\item For each pair of composable morphisms $f:X \to Y$ and $g:Y \to Z$, define the compositor $\quot{\phi_{f,g}}{Q}$ to simply be the restriction of the compositor of $\Ch^b(T^{\heartsuit})$ to $Q$; because it is comprised of isomorphisms it is stable under the formation of $Q$.
\end{itemize}
Because each functor $T(f)$ is exact, we can also apply \cite[{Proposition 4.1.27}]{Monograph} to deduce that there is a localization pseudofunctor $Q^{-1}\Ch^b(T^{\heartsuit}).$ It is routine to check that this is pseudonaturally isomorphic to the pseudofunctor $D^b(T^{\heartsuit})$.

Observe that for each object $X$ of $\Cscr_0$ the functor $H_T \circ \widetilde{\rea}$ carries quasi-isomorphisms of $\Ch^b(T^{\heartsuit})$ to isomorphisms in $T$. Consequently by \cite[{Theorem 4.1.29}]{Monograph} there is a unique pseudonatural transformation
\[
\ul{\rea}:Q^{-1}\Ch^b\left(T^{\heartsuit}\right) \Rightarrow T
\]
for which the diagram
\[
\begin{tikzcd}
\Ch^b(T^{\heartsuit}) \ar[rr]{}{\widetilde{\rea}} \ar[dr, swap]{}{\ul{\lambda}_Q} & & T \\
 & Q^{-1}\Ch^b\left(T^{\heartsuit}\right) \ar[ur, swap]{}{\ul{\rea}}
\end{tikzcd}
\]
commutes strictly in $\Bicat(\Cscr^{\op},\fCat)$. This produces a pseudofunctor realization pseudonatural translation and completes our pseudofunctorification of \cite[{Appendix A}]{Beilinson}, but not our justification of calling it a realization pseudonatural transformation  or the pseudoconification of \cite[{Appendix A}]{Beilinson}. 
\begin{Theorem}\label{Thm: Section Pseudocone Realizations: Pseudofunctor Realization}
Let $F:\Cscr^{\op} \to \fCat$ be a filtered triangulated pseudofunctor over triangulated pseudofunctor $T:\Cscr^{\op} \to \fCat$ and assume that $T$ and $F$ are truncated with compatible $t$-structures. Then if $Q$ is the quasi-isomorphism subpseudofunctor of $\Ch^b(T^{\heartsuit})$, there is a unique pseudonatural transformation $\ul{\rea}:Q^{-1}\Ch^b(T^{\heartsuit}) \to T$ for which the diagram
\[
\begin{tikzcd}
	\Ch^b(T^{\heartsuit}) \ar[rr]{}{\widetilde{\rea}} \ar[dr, swap]{}{\ul{\lambda}_Q} & & T \\
	& Q^{-1}\Ch^b\left(T^{\heartsuit}\right) \ar[ur, swap]{}{\ul{\rea}}
\end{tikzcd}
\]
commutes in $\Bicat(\Cscr^{\op},\fCat)$. Furthermore, each object functor $\quot{\ul{\rea}}{X}$ is $t$-exact for all obects $X$ of $\Cscr$.
\end{Theorem}
\begin{proof}
Everything in the statement of the theorem prior to the word ``furthermore'' was shown prior to stating the theorem. For the final statment, because each category $\quot{Q}{X}^{-1}\Ch^b(T(X)^{\heartsuit})$ is uniquely isomorphic to $D^b(T(X)^{\heartsuit})$, $\quot{\rea}{X}$ is uniquely isomorphic to the realization functor and so the result follows from \cite[{Section A.6}]{Beilinson}.
\end{proof}

Note that because for each object $X$ in $\Cscr$ the $X$-local component of the diagram in Theorem \ref{Thm: Section Pseudocone Realizations: Pseudofunctor Realization} is induced by the diagram
\[
\begin{tikzcd}
\Ch^b(T(X)^{\heartsuit}) \ar[rr]{}{\quot{\widetilde{\rea}}{X}} \ar[dr, swap]{}{\quot{\lambda}{X}} & & T(X)\\
 & \quot{Q^{-1}}{X}\Ch^b(T(X)^{\heartsuit}) \ar[ur, swap]{}{\quot{\rea}{X}}
\end{tikzcd}
\]
where $\quot{\rea}{X}$ is the realization functor constructed in \cite[{Section A.6}]{Beilinson} and $\quot{\widetilde{\rea}}{X}$ is the corresponding chain complex realization. Thus upon taking psuedocones of the diagram in Theorem \ref{Thm: Section Pseudocone Realizations: Pseudofunctor Realization} we get the pseudocone realization functor.

\begin{Theorem}\label{Thm: Section Pseudocone Realization: Pseudocone Realization}
Let $F:\Cscr^{\op} \to \fCat$ be a filtered triangulated pseudofunctor over triangulated pseudofunctor $T:\Cscr^{\op} \to \fCat$ and assume that $T$ and $F$ are truncated with compatible $t$-structures. Then if $Q$ is the quasi-isomorphism subpseudofunctor of $\Ch^b(T^{\heartsuit})$, there is a unique pseudonatural transformation $\ul{\rea}:Q^{-1}\Ch^b(T^{\heartsuit}) \to T$ for which the diagram
\[
\begin{tikzcd}
	\PC(\Ch^b(T^{\heartsuit})) \ar[rr]{}{\PC(\widetilde{\rea})} \ar[dr, swap]{}{\PC(\ul{\lambda}_Q)} & & \PC(T) \\
	& \PC(D^b\left(T^{\heartsuit}\right)) \ar[ur, swap]{}{\PC(\ul{\rea})}
\end{tikzcd}
\]
commutes in $\Bicat(\Cscr^{\op},\fCat)$ with $\PC(\ul{\rea})$ $t$-exact. Furthermore, for any obejct $A$ of $\PC(\Ch^b(T^{\heartsuit}))$,
\[
\PC(\ul{\rea})(A) = \left\lbrace \quot{\ul{\rea}}{X}(\quot{A}{X}) \; : \; X \in \Cscr_0\right\rbrace
\]
where $\quot{\ul{\rea}}{X}$ is the realization $D^b(T(X)^{\heartsuit}) \to T(X).$
\end{Theorem}
\begin{proof}
Because for every object $X$ in $\Cscr$ there is a unique natural isomorphism $\quot{Q^{-1}}{X}\Ch^b(T(X)^{\heartsuit}) \cong D^b(T(X)^{\heartsuit})$ we get a corresponding pseudonatural isomorphism $D^b(T^{\heartsuit}) \cong Q^{-1}\Ch^b(T^{\heartsuit})$. This allows us to construct the listed pseudocone realization functor as well as prove that it is $t$-exact, as each functor $\quot{\ul{\rea}}{X}$ is $t$-exact where
\[
\quot{\rea}{X}:D^b(T^{\heartsuit}) \to T
\]
is the realization functor of \cite{Beilinson}. Finally the statement regarding $\PC(\ul{\rea})$ follows from \cite[{Theorem 4.1.1}]{Monograph}.
\end{proof}
\begin{corollary}\label{Cor: Section Pseudocone Realizations: The equivalence is id on heart}
For all $X$ in $\Cscr$,
\[
\quot{\ul{\rea}}{X}|_{T(X)^{\heartsuit}} = \id_{T(X)^{\heartsuit}}
\]
and
\[
\PC(\ul{\rea})(A) = A
\]
for all objects $A$ of $\PC(T^{\heartsuit})$ as embedded in $\PC(D^b(T^{\heartsuit}))$.
\end{corollary}
\begin{proof}
The first claim follows from the fact that by \cite[Section A.6]{Beilinson} we have $\quot{\ul{\rea}}{X}|_{T(X)^{\heartsuit}} = \id_{T(X)^{\heartsuit}}$ while the second follows from how we calculate the object assignment of $\PC(\ul{\rea})$; see \cite[Theorem 4.1.1]{Monograph} for details.
\end{proof}

\section{Equivariant Beilinson's Theorem and Ext Functors}
	
We now can give a short proof of an equivariant version of Beilinson's theorem by using the technology of Section \ref{Section: Pseudocone Realization Functors} and in particular Theorem \ref{Thm: Section Pseudocone Realization: Pseudocone Realization} above.
	
\begin{Theorem}\label{Thm: Equivariant Beilinson}
The following equivalences of categories hold:
\begin{enumerate}
	\item Let $K$ be a field, $G$ a smooth algebraic group over $K$, and $X$ a left $G$-variety. If $R$ is any finite ring with characteristic coprime to the characteristic of $K$ then there is an equivalence of categories
	\[
	D^b_G(X;\RMod) \simeq D^b_G(\Per(X;\RMod)
	\]
	where $\Per(X;\RMod)$ denotes the category of perverse sheaves with coefficients in $\RMod$, $D^b_G(\Per(X;\RMod))$ is the equivariant derived category of perverse sheaves of $R$-modules.
	\item Let $K$ be an algebraically closed field, $G$ an algebraic group over $K$, and $X$ a left $G$-variety. Then
	\[
	D_G^b(X;\overline{\Q}_{\ell}) \simeq D_G^b(\Per(X;\overline{\Q}_{\ell})).
	\]
	\item If $G$ is a complex algebraic group and $X$ is a left $G$-variety over $\C$ then for any field $K$ there is an equivalence
	\[
	D_G^b(X(\C),\KVect) \simeq D_G^b(\Per(X(\C),\KVect)).
	\]
	\item If $K = \Fbb_q$ is a finite field, $G$ is an algebraic group over $K$, and $X$ is a left $G$-variety then for any prime $\ell$ coprime to $q$,
	\[
	D_{G,m}^b(X,\overline{\Q}_{\ell}) \simeq D_G^b(\Per_{m}(X;\overline{\Q}_{\ell})),
	\]
	where $D_{G,m}^b(X,\overline{\Q}_{\ell})$ is the equivariant derived category of mixed $\ell$-adic sheaves (and similarly in the perverse case).
\end{enumerate}
\end{Theorem}
\begin{proof}
We prove the theorem explicitly only for Case (2); each of the other cases follows mutatis mutandis to the proof we present. Write 
\[
D^b(\Per(G \backslash (-);\overline{\Q}_{\ell})):\SfResl_G(X)^{\op} \to \fCat
\] 
for the pseudofunctor defining $D_G^b(\Per(X;\overline{\Q}_{\ell}))$ and
\[
D^b_c(G \backslash(-); \overline{\Q}_{\ell}):\SfResl_G(X)^{\op} \to \fCat
\]
for the pseudofunctor defining $D_G^b(X;\overline{\Q}_{\ell})$. By Theorems if we can prove that $D^b_G(\Per(-;\overline{\Q}_{\ell}))$ is obtained from a filtered triangulatedd pseudofunctor over the corresponding pseudofunctor determining $D_G^b(X;\overline{\Q}_{\ell})$, then by Theorems \ref{Thm: Section Pseudocone Realizations: Pseudofunctor Realization} and \ref{Thm: Section Pseudocone Realization: Pseudocone Realization} we have a pseudocone realization functor 
\[
\rea_G := \PC(\ul{\rea}):D_G^b(\Per(X);\overline{\Q}_{\ell}) \to D_G^b(X;\overline{\Q}_{\ell}).
\] 
Furthermore, by \cite[{Theorem 1.3}]{Beilinson} we know that each object functor $\quot{\ul{\rea}}{X}$ is an equivalence, so upon promoting it to an adoint equivalence if necessary we get our desired equivalence of categories. Consequently we simply need to verify that $D^b(\Per(G \backslash (-); \overline{\Q}_{\ell}))$ is a filtered triangulated pseudofunctor over $D^b_c(G \backslash (-);\overline{\Q}_{\ell})$ then we are done. 

We first show that $D^b(\Per(G \backslash (-);\overline{\Q}_{\ell})$ is a filtered triangulated pseudofunctor. However, each category $D^b(\Per(G \backslash (\Gamma \times X);\overline{\Q}_{\ell}))$ is filtered by \cite[Example A.2]{Beilinson} and for each morphism $f\times \id_X:\Gamma \times X \to \Gamma^{\prime} \times X$ in $\SfResl_G(X)$ the corresponding pullback functor
\[
\begin{tikzcd}
D^b\left(\Per\left(G \backslash (\Gamma^{\prime} \times X);\overline{\Q}_{\ell}\right)\right) \ar[d]{}{{}^{p}\left(G \backslash(f \times \id_X)\right)^{\ast}} \\
D^b\left(\Per\left(G \backslash(\Gamma \times X);\overline{\Q}_{\ell}\right)\right)
\end{tikzcd}
\]
is an $f$-functor because $G \backslash (f \times \id_X)$ is smooth\footnote{This implies that if $d_{f}$ is the relative dimension of $G \backslash (f \times \id_X)$, then $(G \backslash (f \times \id_X))^{\ast} \cong {}^{p}(G \backslash (f \times \id_X))^{\ast}[-d_f]$ and so is exact at the derived categorical level.}.

We now show that $D^b(\Per(G \backslash (-);\overline{\Q}_{\ell})$ is filterd over $D^b_c(G \backslash (-); \overline{\Q}_{\ell})$. Observe that for each smooth free $G$-variety $\Gamma$, $D^b(\Per(G \backslash (\Gamma \times X);\overline{\Q}_{\ell}))$ is filtered over $D^b_c(G \backslash (\Gamma \times X);\overline{\Q}_{\ell}))$ by \cite{Beilinson}. As such there are equivalences
\[
\begin{tikzcd}
D_c^b\left(G \backslash (\Gamma \times X); \overline{\Q}_{\ell}\right) \ar[d]{}{\quot{i}{\Gamma \times X}} \ar[d, swap]{}{\simeq} \\
D^b\left(\Per\left(G \backslash (\Gamma \times X); \overline{\Q}_{\ell}\right)\right)^{\fil \geq 0} \cap D^b\left(\Per\left(G \backslash (\Gamma \times X);\overline{\Q}_{\ell}\right)\right)^{\fil\leq 0}
\end{tikzcd}
\]
for all objects $\Gamma \times X$ in $\SfResl_G(X)$. We now only need to verify that these equivalences vary pseudonaturally in $\SfResl_G(X)^{\op}$.

Fix an $\SfResl_G(X)$ morphism $f \times \id_X:\Gamma \times X \to \Gamma^{\prime} \times X$. Because each morphism
\[
G \backslash (f \times \id_X):G\backslash(\Gamma \times X) \rightarrow G \backslash (\Gamma^{\prime} \times X)
\]
is smooth it follows that the ${}^{p}(G \backslash (f \times \id_X))^{\ast}$ are $f$-functors and so preserve the filtered degree $0$ component. Consequently there is a natural isomorphism $\quot{i}{f \times \id_X}$ of the form
\[
\quot{i}{\Gamma \times X} \circ \left(G \backslash (f \times \id_X)\right)^{\ast} \xRightarrow{\cong} {}^{p}\left(G \backslash (f \times \id_X)\right)^{\ast} \circ \quot{i}{\Gamma^{\prime} \times X}
\]
Checking that these vary pseudonaturally is routine and shows that the functors $\quot{i}{\Gamma \times X}$ constitute the object functors of a pseudonatural equivalence:
\[
\begin{tikzcd}
D^b_c\left(G \backslash (-); \overline{\Q}_{\ell}\right) \ar[d]{}{\ul{i}} \\
D^b\left(\Per(G \backslash (-);\overline{\Q}_{\ell})\right)^{\fil\geq 0} \cap D^b\left(\Per(G \backslash (-); \overline{\Q}_{\ell})\right)^{\fil\leq 0}
\end{tikzcd}
\]
Thus we conclude that $D^b(\Per(G \backslash (-);\overline{\Q}_{\ell}))$ is filtered triangulated over the pseudofunctor $D^b_c(G \backslash (-);\overline{\Q}_{\ell})$. From here, because we have shown that the pseudofunctor $D^b(\Per(G \backslash (-); \overline{\Q}_{\ell}))$ is filtered triangulated over $D^b_c(G \backslash (-);\overline{\Q}_{\ell})$, the theorem follows by the process described at the start of the proof.
\end{proof}
\begin{remark}
We use the definition of mixed sheaves as given in \cite{BBD}, as this is the setting that \cite{Beilinson} proved the corresponding equivalence of categories. While mixed sheaves have been defined for varieties over non-finite fields (cf., for instance, \cite{Huber}) and the technique here does likely carry over to this case, it is beyond the scope and aim of this paper to do this extension here. 
\end{remark}
\begin{remark}
Similarly to the mixed sheaf case, we only consider the categories $D_G^b(X;\overline{\Q}_{\ell})$ for varieties $X$ over algebraically closed fields $X$; this is simply so as to line up with Beilinson's \cite{Beilinson}[{Theorem 1.3}]. The techniques used in the proof below carry over to the case where $X$ is a variety over any field $K$ (by using the extension of Deligne's category of $\ell$-adic sheaves and $\ell$-adic perverse sheaves to arbitrary varities by way of \cite{Ekedahl}) and give an equivalence $D_G^b(X;\overline{\Q}_{\ell}) \simeq D^b_G(\Per(X;\overline{\Q}_{\ell}))$ provided we assume that $G$ is a smooth algebraic group.
\end{remark}
\begin{remark}
The reader familiar with \cite{Beilinson} will note that we are missing the equivariant analogue of the last statement of \cite[Theorem 1.3]{Beilinson}: the (suggested) equivalence of the equivariant derived category of equivariant holonomic $D$-modules and the equivariant derived category of perverse sheaves for a characteristic zero field. While I conjecture such an equivalence holds and the proof of Theorem \ref{Thm: Equivariant Beilinson} may be adapted to even prove ssaid euivalence, it is nothing more than my ignorance of the $D$-module formalism and the nature or existence of a candidate pseudofunctor of derived categories of complexes of $D$-modules with bounded holonomic cohomology that led to the absence of such a result here.
\end{remark}
\begin{corollary}\label{Cor: Section Equiv Beilin: EDC is EDPer for all var}
Let $K$ be a field, $G$ a smooth algebraic group over $K$, and let $X$ be a left $G$-variety. Then there is an equivalence of categories
\[
D_G^b(X;\overline{\Q}_{\ell}) \simeq D^b_G\left(\Per(X;\overline{\Q}_{\ell})\right).
\]
\end{corollary}
	
Based on the theorems and corollaries above, we can use Theorem \ref{Thm: Equivariant Beilinson} and the description of morphisms in pseudocone categories to give an isomorphism
\[
D^b_{G}(\Per(X;\overline{\Q}_{\ell}))(\Fscr,\Gscr[n]) \cong D^b_G(X;\overline{\Q}_{\ell})(\Fscr,\Gscr[n])
\]
for any two equivariant perverse sheaves $\Fscr$ and $\Gscr$.

\begin{proposition}\label{Prop: The Exty boi}
Let $\Fscr$ and $\Gscr$ be two equivariant perverse sehaves and let $n \in \Z$. Then there is an isomorphism
\[
\Ext^n_{\Per_G(X;\overline{\Q}_{\ell})}(\Fscr,\Gscr) \cong \Ext^n_{D_G^b(X;\overline{\Q}_{\ell})}(\Fscr,\Gscr).
\]
\end{proposition}
\begin{proof}
Fix a length $n$-extension
\[
\begin{tikzcd}
0 \ar[r] & \Gscr \ar[r] & \Pscr_n \ar[r] & \cdots \ar[r] & \Pscr_1 \ar[r] & \Fscr \ar[r] & 0
\end{tikzcd}
\]
of equivariant perverse sheaves. Because we have an equivalence of categories $\Per_G(X;\overline{\Q}_{\ell}) \simeq \PC(\Per(G \backslash (-);\overline{\Q}_{\ell}))$ (cf.\@, for instance, \cite{PramodBook} or \cite{Monograph}) we can regard each perverse sheaf above as an object in a pseudocone category. Consequently, because of how exact sequences are determined in the Abelian category $\PC(\Per(G \backslash(-);\overline{\Q}_{\ell}))$ (cf.\@ \cite[Theorem 3.1.3, Proposition 3.1.12]{Monograph}) we have, for all objects $\Gamma \times X$ of $\SfResl_G(X)$, exact sequences
\[
\begin{tikzcd}
	0 \ar[r] & \quot{\Gscr}{\Gamma \times X} \ar[r] & \quot{\Pscr_n}{\Gamma \times X} \ar[r] & \cdots \ar[r] & \quot{\Pscr_1}{\Gamma \times X} \ar[r] & \quot{\Fscr}{\Gamma \times X} \ar[r] & 0
\end{tikzcd}
\]
of perverse sheaves in $\Per(G \backslash (\Gamma \times X);\overline{\Q}_{\ell})$. Note also that these are compatible with the transition isomorphisms $\tau_f$ in the sense that for all morphisms $f \times \id_X:\Gamma \times X \to \Gamma^{\prime} \times X$ we have commuting diagrams
\[
\begin{tikzcd} \\
 {}^{p}(G \backslash (f \times \id_X))^{\ast}(\quot{\Gscr}{\Gamma^{\prime} \times X}) \ar[d, swap]{}{\tau_f^{\Gscr}} \ar[r, swap] & \cdots \ar[r] & {}^{p}(G \backslash (f \times \id_X))^{\ast}(\quot{\Fscr}{\Gamma^{\prime} \times X}) \ar[d]{}{\tau_f^{\Fscr}} \\
	\quot{\Gscr}{\Gamma \times X} \ar[r] & \cdots \ar[r] & \quot{\Fscr}{\Gamma \times X} 
\end{tikzcd}
\]
with each induced square a commuting square; note that the $0$'s are omitted from the diagram so the image fits within the page margins. We then can translate each such $n$-extension to a morphism
\[
\quot{\rho}{\Gamma \times X}:\Fscr \to \quot{\Gscr[n]}{\Gamma}
\]
in $D^b(\Per(G \backslash (\Gamma \times X);\overline{\Q}_{\ell}))$ and each such morphism arises from a different isomorphism class of extensions. Now note that the maps $\tau_f^{\Fscr}$ and $\tau_f^{\Gscr}$ then give rise to isomorphisms making each corresponding diagram
\[
\begin{tikzcd}
{}^{p}(G \backslash (f \times \id_X))^{\ast}(\quot{\Fscr}{\Gamma^{\prime} \times X})  \ar[d, swap]{}{\tau_f^{\Fscr}} \ar[rrrr, swap]{}{{}^{p}(G \backslash (f \times \id_X))^{\ast}(\quot{\rho}{\Gamma^{\prime} \times X})} & & & & {}^{p}(G \backslash (f \times \id_X))^{\ast}(\quot{\Gscr}{\Gamma^{\prime} \times X}[n]) \ar[d]{}{\tau_f^{\Gscr[n]}} \\
\quot{\Fscr}{\Gamma \times X} \ar[rrrr, swap]{}{\quot{\rho}{\Gamma}} & & & & \quot{\Gscr[n]}{\Gamma}
\end{tikzcd}
\]
commute in $D^b(\Per(G \backslash (\Gamma \times X);\overline{\Q}_{\ell}))$. But this in turn means that we have a morphism $P = \lbrace \quot{\rho}{\Gamma \times X} \; | \; \Gamma \times X \in \SfResl_G(X)_0 \rbrace$ from $\Fscr$ to $\Gscr[n]$ and so, by applying the realization equivalence of Theorem \ref{Thm: Equivariant Beilinson} we get an isomorphism
\[
D_G^b\left(\Per(X;\overline{\Q}_{\ell})\right)(\Fscr,\Gscr[n]) \cong D^b_G(X;\overline{\Q}_{\ell})(\PC(\rea)(\Fscr), \PC(\rea)(\Gscr)[n]).
\]
This further implies that, because $\quot{\ul{\rea}}{X}(\quot{\Fscr}{\Gamma \times X}) = \quot{\Fscr}{\Gamma \times X}$ by Corollary \ref{Cor: Section Pseudocone Realizations: The equivalence is id on heart}, for each object $\Gamma \times X$ of $\SfResl_G(X)$, the existence of a morphism 
\[
\quot{\widetilde{\rho}}{\Gamma \times X}:\quot{\Fscr}{\Gamma \times X} \rightarrow \quot{\Gscr[n]}{\Gamma \times X}
\] 
in $D_c^b(G \backslash (\Gamma \times X); \overline{\Q}_{\ell})$ which is suitably compatible with the transition isomorphisms of $\Fscr$ and $\Gscr[n]$. This further implies, upon once again walking back the derived-hom/extension correspondence, that such a family of compatible morphisms gives rise to compatible exact sequences
\[
\begin{tikzcd} \\
	(G \backslash (f \times \id_X))^{\ast}(\quot{\Gscr}{\Gamma^{\prime} \times X}) \ar[d, swap]{}{\tau_f^{\Gscr}} \ar[r, swap] & \cdots \ar[r] & (G \backslash (f \times \id_X))^{\ast}(\quot{\Fscr}{\Gamma^{\prime} \times X}) \ar[d]{}{\tau_f^{\Fscr}} \\
	\quot{\Gscr}{\Gamma \times X} \ar[r] & \cdots \ar[r] & \quot{\Fscr}{\Gamma \times X} 
\end{tikzcd}
\]
of perverse sheaves as embedded in the usual constructible derived category of the quotient variety. Upon lifting this to an object in the pseudocone category we then get an $n$-extension
\[
\begin{tikzcd}
	0 \ar[r] & \Gscr \ar[r] & \Pscr_n \ar[r] & \cdots \ar[r] & \Pscr_1 \ar[r] & \Fscr \ar[r] & 0
\end{tikzcd}
\]
as computed in $D_G^b(X;\overline{\Q}_{\ell})$. This process may be reversed and gives rise to the claimed isomorphism.
\end{proof}

\printbibliography
	
\end{document}